\documentclass{article}
\usepackage{graphicx} 
\usepackage{indentfirst}
\usepackage{geometry}
\usepackage{chngpage}

\usepackage{ulem}    

\usepackage{amssymb,amsfonts,amsmath,amsthm,mathtools}
\usepackage{enumitem}
\usepackage{xcolor}
\usepackage{mdframed}
\mdfdefinestyle{shadeStyle}{backgroundcolor=white,linewidth=0pt,innerleftmargin=1ex,innertopmargin=-0.5ex}
\theoremstyle{plain}
\newmdtheoremenv[style=shadeStyle]{theorem}{Theorem}[section]
\newmdtheoremenv[style=shadeStyle]{lemma}[theorem]{Lemma}
\newmdtheoremenv[style=shadeStyle]{proposition}[theorem]{Proposition}
\newmdtheoremenv[style=shadeStyle]{corollary}[theorem]{Corollary}
\newmdtheoremenv[style=shadeStyle]{remark}[theorem]{Remark}

\renewcommand{\abstract}{Abstract.}
\theoremstyle{definition}
\newcommand{\rd}{\mathrm{d}}
\usepackage{amsmath}
\numberwithin{equation}{section}

\usepackage[T1]{fontenc}
\usepackage{authblk}

\begin{document}
	\title{\bf On a Keller--Segel System with Density-Suppressed Motility, Indirect Signal Production, and External Sources}
	\author[1,2]{Yujiao Sun 
		\thanks{sunyujiao24@mails.ucas.ac.cn}}
	\author[1,2]{Jie Jiang 
		\thanks{Corresponding author: jiang@apm.ac.cn}}
	\affil[1]{Innovation Academy for Precision Measurement Science and Technology, Wuhan Institute of Physics and Mathematics, Chinese Academy of Sciences, Wuhan 430071, China}
	\affil[2]{University of Chinese Academy of Sciences, Beijing 100049, China}
	\date{\today}
	\maketitle
	
	\begin{abstract}
		This paper investigates an initial-Neumann boundary value problem for a  Keller--Segel system with parabolic-parabolic-ODE coupling. The model incorporates a signal-dependent, non-increasing motility function that,  through indirect signal production, captures a self-trapping effect suppressing cellular movement at high densities.
		We establish the global existence of classical solutions in arbitrary spatial dimensions for a broad class of non-increasing motility functions, both with and without external source terms. Furthermore, we demonstrate that any external damping source exhibiting superlinear growth ensures uniform-in-time boundedness. Conversely, in the absence of such damping, solutions may become unbounded as time tends to infinity. More precisely, in the two-dimensional homogeneous case with the exponentially decaying motility function 
		$\gamma(v) = e^{-v}$
		, a critical mass phenomenon emerges: classical solutions remain  uniformly bounded for subcritical initial mass, while supercritical initial masses can lead to infinite-time blow-up.
		Our analysis relies on the construction of carefully designed auxiliary functions along with refined comparison methods and iteration arguments.
		\textit{2020 MR Subject Classification:} 35K51; 35K57; 35K59; 35M13; 35Q92 
	\end{abstract}
	
	\section{Introduction}
	Chemotaxis describes the directed movement of biological cells or organisms along chemical concentration gradients.  This process is fundamental to numerous biological phenomena, such as bacterial aggregation, immune responses, embryonic development, and tumor invasion. In 1971, Keller and Segel established the seminal Keller--Segel model \cite{keller1971model} to describe the aggregation of Dictyostelium discoideum. The model consists of the following equations:
	\begin{equation}\label{Classical KS Eqs}
		\left\{
		\begin{aligned}
			&u_t = \nabla\cdot(\gamma(v)\nabla u - \chi(v) u\nabla v), \\
			&\tau v_t = \Delta v - v + u, 
		\end{aligned}
		\right.
	\end{equation}
	with $\tau \geq 0$. Here, $u$ denotes the cell density, and $v$ stands for chemical signal concentration. The functions $\gamma(\cdot)$ and $\chi(\cdot)$ represent the diffusivity and chemo-sensitivity, respectively. In particular, they satisfy $\chi(\cdot) = (\alpha - 1)\gamma'(\cdot)$, where $\alpha\geq 0$ is a rescaled distance between cellular signal receptors. The special case $\alpha=0$ corresponds to a cell with a single receptor, detecting the signal at one point. Under this assumption, system \eqref{Classical KS Eqs} simplifies to
	\begin{equation}\label{Binary system}
		\left\{
		\begin{aligned}
			&u_t = \Delta(\gamma(v)u), \\
			&\tau v_t = \Delta v - v + u.
		\end{aligned}
		\right.
	\end{equation}
	Note that the monotonicity of the motility function $\gamma$ governs the influence of chemical stimuli on cell movement. Specifically, $\gamma'(\cdot) < 0$ models chemo-attraction, with cells moving toward higher signal concentrations. Conversely, $\gamma'(\cdot) > 0$ models chemo-repulsion, with cells moving away from regions of higher chemical concentration. 
	
	As for the initial-boundary value problem of system \eqref{Binary system} with no-flux boundary and  suitably regular initial data, numerous results regarding the existence of global solutions have been established (see, e.g., \cite{burger2021delayed,fujiejiang2021CVPDEcomparison,fujieSenba2022globalNA,Funaki2012NHM,jiangLaurencot2021JDEglobal,jin2020critical,tao2017effects,xiaojiang2023JDEglobal,yoon2017global}), as summarized in the review \cite{winkler2025effects}. Moreover,  boundedness results have been proved under various structural assumptions. Under the condition that $\gamma$ is uniformly bounded from above and below by positive constants,  uniform boundedness of classical solutions is proved for arbitrary spatial dimensions in \cite{xiaojiang2023JDEglobal}, which improves the earlier result in \cite{tao2017effects} for two-dimensional convex domains. In the absence of a strictly positive lower bound for the motility function, boundedness is shown to be related to the decay rate of the motility function at infinity. Indeed, for system \eqref{Binary system} in two dimensions, classical solutions are known to remain globally bounded if $\gamma$ decays slower than exponentially \cite{fujiejiang2021AAMboundedness,jiangLaurencot2022CPDEglobal,xiaojiang2023JDEglobal}. A notable exception is the exponentially decaying motility $\gamma(v)=e^{- v}$. In this case, system \eqref{Binary system} shares the same mathematical features with the Keller--Segel system, such as the Lyapunov functional and stationary problem, but exhibits different mass-critical phenomena \cite{fujiejiang2020JDEglobal,fujiejiang2021CVPDEcomparison,jin2020critical}. Specifically, unlike the finite-time blow-up known for the classical Keller--Segel equations, here certain initial data beyond a threshold mass yield solutions that blow up in infinite time. In higher dimensions $N\geq3$, boundedness is examined for the power-law motility $\gamma(s)=s^{-k}$ with $k\in (0,\frac{N}{(N-2)^+})$, see \cite{ahn2019global,fujieSenba2022globalNonlinearity,jiangLaurencot2021JDEglobal}.

	
	In addition, the nonhomogeneous counterpart of system \eqref{Binary system} has also attracted considerable attention. To study ecological pattern formation, Fu et al. \cite{fu2012stripe} introduce the following model with logistic growth:
	\begin{equation}
		\left\{
		\begin{aligned}
			&u_t = \Delta(\gamma(v)u) + \mu u(1-u),\\
			&\tau v_t = \Delta v - v + u,
		\end{aligned}
		\right.
	\end{equation}
	where $\mu > 0$. Globally bounded classical solutions have been proved under various conditions. In the parabolic–elliptic case ($\tau = 0$) with $N=2$, Jin and Wang \cite{jin2021keller}  prove uniform-in-time boundedness under the assumption that $\sup\limits_{0 \leq s < \infty} \frac{|\gamma'(s)|^2}{\gamma(s)} < \infty$, which is improved subsequently by Fujie and Jiang \cite{fujiejiang2020JDEglobal} by removing this condition and requiring only that $\gamma'\leq 0$. For the parabolic–parabolic system ($\tau > 0$) in two dimensions, boundedness is established by Jin, Kim and Wang \cite{jin2018boundedness} assuming $\gamma'<0$, $\lim\limits_{s\to\infty}\gamma(s)=0$ and $\frac{|\gamma'(s)|}{\gamma(s)} < \infty$. In three dimensions ($N = 3$), if $\gamma$ has positive upper and lower bounds and $|\gamma'|$ is bounded, Liu and Xu \cite{liu2019large} obtain the same conclusion for sufficiently large $\mu$. For $N \geq 3$, under the assumptions $\gamma' \leq 0$, $\lim\limits_{s \to \infty} \gamma(s) = 0$, and $|\gamma'|$ being bounded, Wang and Wang \cite{wang2019boundedness} also prove the existence of globally bounded solutions for sufficiently large $\mu$.
	
	When the standard logistic nonlinearity $\mu u(1 - u)$ is generalized to  $-u f(u)$, additional boundedness results are available in the literature. For $\tau = 0$, $\gamma' \leq 0$ and $-u f(u) = \mu(u - u^k)$ with $k > 1$, Lyu and Wang \cite{lyu2022logistic} demonstrate that the solutions remain boundedness if one of the following conditions is satisfied: (i) $N \leq 2$ for any $\mu>0$ and $k > 1$; (ii) $N \geq 3$ for any $\mu>0$ and $k > 2$; or (iii) $N \geq 3$, $k = 2$ and $\mu$ is sufficiently large. Another case arises when the source term is given by Gompertz-type growth, for which $f$ satisfies $\lim\limits_{s \to \infty} f(s) = \infty$ and $\limsup\limits_{s \to \infty} \frac{f(s)}{\log s} < \infty$. For this nonlinearity, Xiao and Jiang \cite{xiaojiang2024prevention} establish boundedness when $\tau \geq 0$ and $\gamma(s) = e^{-s}$ in dimensions $N \leq3$. Subsequently, Lu and Jiang \cite{lu2025suppression} remove the structural restriction on $\gamma$ and obtain the same conclusion for the parabolic-elliptic system in dimensions $N\leq3$.
	
	The classical Keller--Segel system assumes direct chemoattractant secretion by cells; however, many biologically realistic scenarios involve indirect signal production through intermediate mechanisms. A notable instance is the model for mountain pine beetle (MPB) dispersal and aggregation dynamics in forest habitats, formulated by Strohm et al. \cite{strohm2013pattern}:
	\begin{equation}\label{MPB model}
		\left\{
		\begin{aligned}
			& u_t = \Delta u-\nabla\cdot(u\nabla v)+\mu u-\mu u^l, \\
			& \tau v_t=\Delta v-v+h,  \\
			& h_t+\delta h=u, 
		\end{aligned}
		\right.
	\end{equation}
	where $\mu>0$, $\delta>0$, $\tau\geq0$ and $l>1$. Here $u$ and $h$ represent the densities of flying and nesting MPBs, respectively, while $v$ denotes the concentration of beetle pheromone. 
	For $\tau>0$, the global existence and uniform boundedness of classical solutions for $N\geq2$  have been established in \cite{hu2016exclusion,li2018boundedness} provided $l>\frac{N}{2}$, with the case $l = \frac{N}{2}$ later covered for $N\geq3$ \cite{ren2021boundedness}. For the parabolic-elliptic counterpart ($\tau=0$) with quadratic degradation ($l=2$), global existence holds \cite{ye2022boundedness} either in low dimensions $N\leq 2$, or in higher dimensions $N\geq 3$ with $\mu$ being sufficiently large. Without the logistic damping ($\mu=0$), Lauren\c{c}ot \cite{laurencot2019DCDSBglobal} establishes global solvability and demonstrates an infinite-time critical mass blow-up phenomenon in two dimensions. Subsequently, Soga \cite{soga2025JEEconcentration} analyzes radially symmetric solutions in a disk and concludes that infinite-time blow-up occurs solely at the origin.  We refer to \cite{MIAO2023JDE,Ren2022JMP} for related models featuring indirect signal production.
	
	Combining features of \eqref{Binary system} and \eqref{MPB model}, Lv and Wang \cite{lv2020global} study a chemotaxis system with signal-dependent motility, indirect signal production, and a generalized logistic source:
	\begin{equation}\label{Lv's model}
		\left\{
		\begin{aligned}
			& u_t = \Delta(\gamma(v)u)+\lambda u-\mu u^l, \\
			& v_t=\Delta v-v+h,  \\
			& h_t+\delta h=u,  
		\end{aligned}
		\right.
	\end{equation}
	with $\lambda\in\mathbb{R}$, $\mu>0$, and $\delta>0$. Under the assumptions $\gamma'(\cdot) < 0$, $\sup\limits_{0 \leq s < \infty} \frac{|\gamma'(s)|}{\gamma(s)} < \infty$ and $\ l > \max\left\{1, \frac{N}{2}\right\}$, they establish the global existence and boundedness of classical solutions to system \eqref{Lv's model}.

	In this paper, we study the following initial-Neumann boundary value problem in a smooth bounded
	domain $\Omega\subset\mathbb{R}^N$ with $N\geq1$.
	\begin{subequations}\label{Our Problem}
		\begin{align}
			& u_t = \Delta \left( \gamma(v) u\right) - uf(u),  &(t,x) &\in  (0,\infty) \times\Omega,\label{Eq of u}\\
			& \tau v_t=\Delta v-v+h, &(t,x) &\in  (0,\infty) \times\Omega,\label{Eq of v}\\
			& \delta h_t+h=u, 
			&(t,x) &\in  (0,\infty) \times\Omega,\label{Eq of h}\\
			&\nabla u\cdot\mathbf{n} = \nabla v\cdot\mathbf{n} =0, 
			&(t,x) &\in (0,\infty)\times\partial\Omega,\label{boundcond}\\
			&\left(u,v,h\right)(0,x)=(u_0,v_0,h_0),  
			&x &\in  \Omega,
		\end{align}
	\end{subequations}
	where $\tau>0$, $\delta>0$, and $\mathbf{n}$ denotes the outward unit normal to the boundary.
	We assume that $(u_0,v_0,h_0)$ satisfies the following conditions:
	\begin{equation}\label{A0}
		u_0,\,v_0\in W^{2,\infty}(\Omega),\,h_0\in W^{1,\infty}(\Omega);\quad u_0,\,v_0,\,h_0\geq0,\,u_0\not \equiv 0\,\,\,\text{in}\,\, \Omega;\quad
		\nabla u_0\cdot \mathbf{n} = \nabla v_0\cdot\mathbf{n} =0\,\,\,\text{on}\,\, \partial\Omega. \tag{A0}
	\end{equation}
	For $\gamma(\cdot)$, we require that
	\begin{equation}\label{A1}
		0<\gamma(\cdot)\in C^{3}\left([0,\infty)\right),\quad\gamma'(\cdot)\leq0 \quad \text{on}~~[0,\infty). \tag{A1}
	\end{equation}
	Clearly, $\gamma$ admits a uniform positive upper bound $\gamma^*:=\gamma(0)$. And we denote by $C$ and $C_i$ (with $i$ being a positive integer) generic positive constants that may change from line to line.

	In what follows, we analyze problem \eqref{Our Problem} separately in the homogeneous and nonhomogeneous settings. For the homogeneous problem, we first establish the global existence of classical solutions under the minimal assumption \eqref{A1} on  the motility function. Assuming additionally that $\gamma$ is bounded below by a strictly positive constant, we further prove the uniform boundedness of solutions. Conversely, when such a lower bound is absent, we demonstrate that infinite-time blow-up can occur  in two dimensions for the specific motility function $\gamma(v)=e^{-v}$. Finally, for the nonhomogeneous problem, we prove that any external  source with superlinear damping  guarantees uniform-in-time boundedness of solutions, thereby suppressing the potential blow-up present in the source-free case.
	
	We are now in a position to state our first main result for the homogeneous case of \eqref{Our Problem}. 
	\begin{theorem}\label{Th.1.1}
		Let $\Omega$ be a bounded domain of  $\mathbb{R}^N$($N\geq1$) with smooth boundary. Suppose that the initial data $(u_0,v_0,h_0)$ satisfies condition \eqref{A0}, the motility function $\gamma(\cdot)$ satisfies assumption \eqref{A1} and $f \equiv 0$. Then problem \eqref{Our Problem} has a unique global nonnegative classical solution $(u,v,h)$ such that
		\begin{align*}
			&u \in C^0\left([0,\infty
			)\times\overline{\Omega}\right) \cap C^{1,2}\left((0,\infty)\times\overline{\Omega}\right),\\
			&v \in C^0\left([0,\infty
			)\times\overline{\Omega}\right) \cap C^{1,2}\left((0,\infty)\times\overline{\Omega}\right),\\
			&h\in C^{1}\left([0,\infty); C^{0}(\overline{\Omega})\right).
		\end{align*}
	\end{theorem}
	\begin{remark}
		It is worth noting that our existence result applies to all motility functions satisfying \eqref{A1}, without requiring the additional asymptotic smallness condition $\limsup\limits_{s\rightarrow\infty}\gamma(s)<1/\tau$, which was necessary in the direct signal production framework \cite{fujiejiang2021CVPDEcomparison,jiangLaurencot2022CPDEglobal}.
	\end{remark}
	
	Moreover, under the additional assumption that $\gamma$ admits a strictly positive lower bound, we prove that the solution is uniformly bounded in time.
	
	\begin{theorem}\label{Th.1.2}
		Under the same conditions as those in Theorem \ref{Th.1.1}, we further assume that $\gamma$ satisfies
		\begin{equation}\label{A2}
			\gamma(s) \geq \gamma_* > 0 \qquad \text{for all} \,\, s \geq 0, \tag{A2}
		\end{equation}
		with a generic constant $\gamma_*>0$, then the homogeneous problem of \eqref{Our Problem} admits a global  classical solution $(u,v,h)$ that is uniformly bounded in time. More precisely, there is a constant $C>0$ depending only on $\Omega$, $\delta$, $\tau$, $\gamma$ and the initial data such that
		\begin{equation*}
			\sup\limits_{t \geq 0} \left(\|u(t)\|_{L^\infty(\Omega)} + \|v(t)\|_{L^\infty(\Omega)} + \|h(t)\|_{L^\infty(\Omega)}\right) \leq C.
		\end{equation*}
	\end{theorem}

	Next, we focus on the specific case $\gamma(v)=e^{-v}$, showing a threshold phenomenon.
	\begin{theorem}\label{Th.1.3}
		Let $\Omega\subset\mathbb{R}^2$ be a domain with smooth boundary. Assume that $f \equiv 0$, $\gamma(v) = e^{-v}$, $\tau = \delta = 1$ and $(u_0,v_0,h_0)$ satisfies \eqref{A0}. Let 
		\begin{equation*}
			M =
			\begin{cases}
				8\pi & \text{if }~ \Omega = B_R(0) \triangleq \{x \in \mathbb{R}^2; \ |x| < R\} \text{ and } (u_0, v_0, h_0) \text{ is radial in } x, \\
				4\pi & \text{otherwise}.
			\end{cases}
		\end{equation*}
		Then, 
		\begin{enumerate} [label=(\roman*), leftmargin=*]
			\item if $m \triangleq \int_\Omega u_0\;\rd x < M$, the global classical solution to \eqref{Our Problem} is uniformly bounded in time;
			\item there exists nonnegative initial data satisfying \eqref{A0} with $m\in(8\pi,\infty)\setminus4\pi\mathbb{N}$ such that the corresponding solution to \eqref{Our Problem} blows up at time infinity, i.e.,
			\begin{equation*}
				\limsup\limits_{t\to\infty} \left(\|u(t)\|_{L^\infty(\Omega)} + \|v(t)\|_{L^\infty(\Omega)} +\|h(t)\|_{L^\infty(\Omega)}\right)= \infty.
			\end{equation*}
		\end{enumerate}
	\end{theorem}
	
	For the nonhomogeneous case with the presence of superlinear damping source terms,  we can also prove the global existence and uniform boundedness of classical solutions.
	
	\begin{theorem}\label{Th.1.4}
		Let $\Omega$ be a bounded domain of  $\mathbb{R}^N$($N\geq1$) with smooth boundary. Suppose that the initial data $(u_0,v_0,h_0)$ satisfies condition \eqref{A0}, the motility function $\gamma(\cdot)$ satisfies assumption \eqref{A1}, and that $f(\cdot)$ satisfies 
		\begin{equation}\label{H}
			f(s)\in C^1([0,\infty)),\quad \lim\limits_{s\to \infty}f(s) = +\infty. \tag{H}
		\end{equation}
		Then problem \eqref{Our Problem} admits a unique global nonnegative classical solution that is uniformly bounded in time.
	\end{theorem}
	
	\begin{remark}
		It is worth recalling that the uniform boundedness result in \cite{lv2020global} relies on the hypotheses  $uf(u)=\mu u^l-\lambda u$ with $\ l > \max\left\{1, \frac{N}{2}\right\}$, and $\sup\limits_{0 \leq s < \infty} \frac{|\gamma'(s)|}{\gamma(s)} < \infty$, which notably exclude rapidly decaying motility function (e.g., $\gamma(v)=e^{-v^2}$). By contrast, our result requires no structural assumptions on $\gamma$; instead, an external damping source with mere superlinear growth suffices to ensure uniform‑in‑time boundedness of the solutions in all space dimensions.
	\end{remark}

	\noindent \textbf{Main ideas.}
	Since the signal-dependent motility function may vanish when $v$  becomes unbounded, establishing an upper bound for $v$ is a crucial first step. To this end, we adapt the approach  developed in \cite{jiangLaurencot2024globalarXiv} by introducing suitable auxiliary functions that yield an explicit decomposition structure for $v$. The desired upper bound then follows from comparison arguments.
	
	Once the upper bound of $v$ is obtained, the conventional strategy in literature (e.g., \cite{fujieSenba2022globalNA,jiangLaurencot2021JDEglobal,jiangLaurencot2022CPDEglobal}) proceeds by first deriving a H\"older estimate for $v$, which then permits the application of semigroup theory to establish higher-order regularity of $v$, with estimates for $u$ subsequently recovered through the system coupling.
	
	By contrast,  the present work develops a substantially simplified approach to  bound  $u$ directly. We introduce the key quantity $\varphi\triangleq\gamma(v)u$, which satisfies a parabolic equation dual to that of $u$. Leveraging the aforementioned decomposition structure along with the monotonicity of $\gamma$, we employ comparison argument or energy method to derive the boundedness of $\varphi$, from which the bound for $u$ follows immediately.

	The remainder of this paper is structured as follows. In Section \ref{section 2}, we prove local well-posedness result and recall several useful lemmas. In Section \ref{section 3}, we introduce several auxiliary functions and construct two key identities. In Section \ref{section 4}, we establish the global existence of solutions to the homogeneous case of \eqref{Our Problem}. Then, in Section \ref{section 5}, under the further assumption that $\gamma$ admits a uniform positive lower bound, we demonstrate the uniform-in-time boundedness. In Section \ref{section 6}, we study the special case  $\gamma(v)=e^{-v}$ showing a threshold phenomenon where infinite-time  blow-up may occur for certain large-mass initial data. Finally, in Section \ref{section 7}, we prove that the mere presence of a superlinear source term suffices to guarantee the global boundedness of solutions.

	\section{Preliminary}\label{section 2}
	We first state the local existence and uniqueness of classical solutions to system \eqref{Our Problem}. Since the proof follows from the standard fixed-point argument and the regularity theory for parabolic equations, see e.g., \cite{jin2018boundedness,tao2011chemotaxis}, we omit the details here.
	
	\begin{theorem}\label{Th.2.1}
		Let $\Omega \subset \mathbb{R}^N$ with $N\geq1$ be a smooth bounded domain. Suppose $0<\gamma(\cdot)\in C^3([0,\infty))$ and $f(\cdot)\in C^1([0,\infty))$. Then for any given initial data $(u_0,v_0,h_0)$ satisfying \eqref{A0}, problem \eqref{Our Problem} admits a unique nonnegative classical solution
		\begin{align*}
			&u\in C^{0}\left([0,T_{\max})\times\overline{\Omega}\right) \cap  C^{1,2}\left((0,T_{\max})\times\overline{\Omega}\right),\\
			&v\in C^{0}\left([0,T_{\max})\times\overline{\Omega}\right) \cap C^{1,2}\left((0,T_{\max})\times\overline{\Omega}\right),\\
			&h\in C^{1}\left([0,T_{\max});C^{0}(\overline{\Omega})\right),
		\end{align*}
		with $T_{\max}\in(0,\infty]$. If $T_{\max} < \infty,$ then
		\begin{equation}\label{blow-up at Tmax}
			\lim\limits_{t \to T_{\max}^-}  \| u(t,\cdot) \|_{L^\infty(\Omega)}  = \infty.
		\end{equation}
	\end{theorem}\

	For the homogeneous case ($f \equiv 0)$, it is directly observed from the equations that system \eqref{Our Problem} satisfies mass conservation, and the $L^1$-norms of $v$ and $h$ remain bounded.
	\begin{lemma}\label{Lemma:u,v,h L1-estimates}
		Assume that $f\equiv0$. Then the classical solution $(u,v,h)$ to \eqref{Our Problem} satisfies mass conservation
		\begin{equation}\label{u: mass conservation}
			\int_\Omega u \;\rd x = \int_\Omega u_0 \;\rd x \quad \text{for all }  t\in[0,T_{\max}).
		\end{equation}
		Moreover, there holds
		\begin{equation}\label{h: L1}
			\|h\|_{L^1(\Omega)} \leq \max\left\{\|h_0\|_{L^1(\Omega)},\, \|u_0\|_{L^1(\Omega)}\right\}
			\quad \text{for all }  t\in[0,T_{\max}),
		\end{equation}
		\begin{equation}\label{v: L1}
			\|v\|_{L^1(\Omega)} \leq \max\left\{\|v_0\|_{L^1(\Omega)},\, \|h_0\|_{L^1(\Omega)},\, \|u_0\|_{L^1(\Omega)} \right\}
			\quad \text{for all }  t\in[0,T_{\max}).
		\end{equation}
	\end{lemma}
	\begin{proof}
		In view of the boundary condition \eqref{boundcond}, a direct integration of \eqref{Eq of u} on $\Omega$ yields \eqref{u: mass conservation} .
		Similarly, we deduce that
		$$\delta\dfrac{d}{dt}\int_\Omega h \;\rd x + \int_\Omega h \;\rd x = \int_\Omega u \;\rd x,$$
		which yields that
		$$\|h\|_{L^1(\Omega)} = e^{-\frac{t}{\delta}} \|h_0\|_{L^1(\Omega)}  + (1-e^{-\frac{t}{\delta}}) \|u\|_{L^1(\Omega)}.$$
		In the same manner, we infer that
		$$\|v\|_{L^1(\Omega)} = e^{-\frac{t}{\tau}} \|v_0\|_{L^1(\Omega)}  + (1-e^{-\frac{t}{\tau}}) \|h\|_{L^1(\Omega)} ,$$
		Thus, \eqref{h: L1} and \eqref{v: L1} follow immediately. This completes the proof.
	\end{proof}
	
	The following lemma provides the lower estimates for $sf(s)$ for all $s\geq 0$.
	\begin{lemma}\label{Lemma: f(s) inequality}
		Suppose that $f$ satisfies the hypothesis \eqref{H} in Theorem \ref{Th.1.4}. Then for any $\varepsilon>0$, there exists a constant $C(\varepsilon)>0$ depending only on $\varepsilon$ such that 
		\begin{equation}\label{f(s) inequality}
			s\leq \varepsilon s f(s)+C(\varepsilon) \quad\quad \text{for all}\ s\geq 0.
		\end{equation}
	\end{lemma}
	\begin{proof}
		Since $f(\cdot)$ satisfies \eqref{H}, we can infer that for any $\varepsilon>0$, there exists $s_\varepsilon>1$ such that $f(s)\geq \frac{1}{\varepsilon}$ for all $s\geq s_\varepsilon$, which implies $s\leq \varepsilon sf(s)$. 
		For $0\leq s<s_\varepsilon$, note that
		$$C(\varepsilon):=\max\limits_{0\leq s\leq s_\varepsilon}|s-\varepsilon s f(s)|<\infty$$ 
		due to the continuity of $f$, which completes the proof. 
	\end{proof}
	
	The following 2D Moser-Trudinger inequality \cite[Proposition 2.3]{C&Y1988JDG} \cite[Theorem 2.1]{NagaiSenbaYoshida1997Application} is needed when we derive a lower bound for the Lyapunov functional.
	
	\begin{lemma}\label{Lemma: Moser-Trudinger inequality}
		Suppose $\Omega$ is a smooth bounded domain in $\mathbb{R}^2$. There exists $K_0 > 0$ depending only on $\Omega$ such that, for $z \in W^{1,2}(\Omega)$,
		\begin{equation}\label{Moser-Trudinger inequality}
			\int_{\Omega} e^{|z|} \, \mathrm{d}x \leq K_0 \exp\left( \frac{\|\nabla z\|_{L^2(\Omega)}^2}{8\pi} + \frac{\|z\|_{L^1(\Omega)}}{|\Omega|} \right).
		\end{equation}
		In particular, when $\Omega=\{x\in\mathbb{R}^2;|x|< R\}$ and $z \in W^{1,2}(\Omega)$ with $z(x)=z(|x|)$, for any $\varepsilon>0$, there exists $K_\varepsilon>0$ depending on $\varepsilon$ and $\Omega$ such that
		\begin{equation}\label{Moser-Trudinger inequality in a disk}
			\int_{\Omega} e^{|z|} \, \mathrm{d}x \leq K_\varepsilon \exp\left( (\frac{1}{16\pi} + \varepsilon)\|\nabla z\|_{L^2(\Omega)}^2 + \frac{2\|z\|_{L^1(\Omega)}}{|\Omega|} \right).
		\end{equation}
	\end{lemma}
	
	In addition, we need the following uniform Gronwall inequality \cite[Chapter III, Lemma
	1.1]{temam1988infinite} to deduce uniform-in-time estimates for the solutions.
	\begin{lemma}\label{Lemma: uniform Gronwall inequality}
		Let $g$, $h$, $y$ be three positive locally integrable functions on $(t_0,\infty)$ such that $y'$ is locally integrable on $(t_0,\infty)$ and the following inequalities are satisfied:
		\begin{equation*}
			\begin{aligned}
				&y'(t) \leq g(t)y(t) + h(t)\quad \text{for } t \geq t_0,\\
				&\int_t^{t+r}g(s)\;\rd s \leq a_1,\,\,\int_t^{t+r}h(s)\;\rd s \leq a_2,\,\,\int_t^{t+r}y(s)\;\rd s \leq a_3 \quad \text{for } t \geq t_0,
			\end{aligned}
		\end{equation*}
		where $r$, $a_i$ ($i=1,2,3$) are positive constants. Then
		\begin{equation*}
			y(t+r) \leq \left(\frac{a_3}{r} + a_2\right)e^{a_1}, \quad \forall\, t \geq t_0.
		\end{equation*}
	\end{lemma}
	
	\section{Introduction of auxiliary functions}\label{section 3}
	To establish the upper bound for $v$, we  introduce several auxiliary functions based on the approach initially proposed in \cite{fujiejiang2021CVPDEcomparison} and recently refined in \cite{jiangLaurencot2024globalarXiv}.
	
	To begin with, we define	
	\begin{equation*}
		\begin{aligned}
			&D(\mathcal{A}):=\left\{z\in H^2(\Omega) : \nabla z\cdot \mathbf{n}=0\,\,\text{on }\partial \Omega\right\},\\
			&\mathcal{A}[z]:=z-\Delta z,\quad z\in D(\mathcal{A}).
		\end{aligned}
	\end{equation*}
	Here, $\Delta$ denotes the usual Laplace operator supplemented with homogeneous Neumann boundary conditions. Recall that $\mathcal{A}$ generates an analytic semigroup on $L^p(\Omega)$ and is invertible on $L^p(\Omega)$ for all $p\in(1,\infty)$. We then set
	\begin{equation}\label{Def: w}
		w(t):=\mathcal{A}^{-1}[u(t)]\geq0, \quad t\in[0,T_{\max}),     
	\end{equation}
	and
	\begin{equation}\label{Def: eta}
		\eta(t):=\mathcal{A}^{-1}[h(t)]\geq0, \quad t\in[0,T_{\max}), 
	\end{equation}
	where the nonnegativity of $w$ and $\eta$ being a consequence of that of $u$ and $h$,  and the elliptic comparison principle. Due to the time continuity of $u$ and $h$, 
	\begin{equation}\label{w_0,eta_0}
		w_0:=w(0)=\mathcal{A}^{-1}[u_0],\quad  \eta_0:=\eta(0)=\mathcal{A}^{-1}[h_0],
	\end{equation}
	and it follows from \eqref{A0} that $w_0\in W^{3,\infty}(\Omega)$ and $\eta_0\in C^2(\overline{\Omega})$. Moreover, it follows from the elliptic comparison principle that
	\begin{equation*}
		\|w_0\|_{L^\infty(\Omega)} = \|\mathcal{A}^{-1}[u_0]\|_{L^\infty(\Omega)} \leq \|u_0\|_{L^\infty(\Omega)},
	\end{equation*}and
	\begin{equation*}
		\|\eta_0\|_{L^\infty(\Omega)} = \|\mathcal{A}^{-1}[h_0]\|_{L^\infty(\Omega)} \leq \|h_0\|_{L^\infty(\Omega)}.
	\end{equation*}
	For convenience, we denote
	\begin{equation}\label{Def: varphi & G(u)}
		\varphi := u\gamma(v),\quad 
		G(u) :=  \mathcal{A}^{-1}[uf(u)],
	\end{equation}
	and infer from the nonnegativity of $u$ and $\gamma$ that $\varphi \geq 0$.
	
	Next, let $\Psi$  be the unique solution to
	\begin{equation}\label{Psi solves}
		\left\{
		\begin{aligned}
			&\mathcal{L}[\Psi]:=\tau\Psi_t-\Delta \Psi+\Psi = \varphi, 
			&(t,x) &\in (0,T_{\max}) \times \Omega,\\
			&\nabla \Psi \cdot \mathbf{n} = 0,
			&(t,x) &\in (0,T_{\max}) \times \partial\Omega,\\
			&\Psi(0) = 0,
			&x &\in \Omega.  
		\end{aligned}
		\right.
	\end{equation}
	By the parabolic comparison principle and the nonnegativity of $\varphi$, we infer that $\Psi(t,x) \geq 0$ for $(t,x) \in (0,T_{\max}) \times \Omega$. Then denote $\psi \triangleq\mathcal{A}^{-1}[\Psi] \geq 0$ and we can immediately deduce from \eqref{Psi solves} that $\psi$ solves
	\begin{equation}\label{psi solves}
		\left\{
		\begin{aligned}
			&\mathcal{L}[\psi] = \mathcal{A}^{-1}[\varphi], 
			&(t,x) &\in (0,T_{\max}) \times \Omega,\\
			&\nabla \psi \cdot \mathbf{n} = 0,
			&(t,x) &\in (0,T_{\max}) \times \partial\Omega,\\
			&\psi(0) = 0,
			&x &\in \Omega.  
		\end{aligned}
		\right.
	\end{equation}
	In the same manner, we define $g$ as the unique solution to
	\begin{equation}\label{g solves}
		\left\{
		\begin{aligned}
			&\mathcal{L}[g] = G(u) = \mathcal{A}^{-1}[uf(u)], 
			&(t,x) &\in (0,T_{\max}) \times \Omega,\\
			&\nabla g \cdot \mathbf{n} = 0,
			&(t,x) &\in (0,T_{\max}) \times \partial\Omega,\\
			&g(0) = 0,
			&x &\in \Omega.  
		\end{aligned}
		\right.
	\end{equation}
	Recall that $v$ solves $\mathcal{L}[v] = h$ in $(0,T_{\max}) \times \Omega$ with $\nabla v \cdot \mathbf{n} = 0$ on $\partial\Omega$ and $v(0) = v_0$.
	We infer by differentiating the equation with respect to time that
	\begin{equation*}
		\left\{
		\begin{aligned}
			&\mathcal{L}[v_t] = h_t, 
			&(t,x) &\in (0,T_{\max}) \times \Omega,\\
			&\nabla v_t \cdot \mathbf{n} = 0,
			&(t,x) &\in (0,T_{\max}) \times \partial\Omega,\\
			&v_t(0) = \frac{1}{\tau}\left(h_0 + \Delta v_0 - v_0\right),
			&x &\in \Omega.  
		\end{aligned}
		\right.
	\end{equation*}
	Then we conclude that $\delta v_t + v$ satisfies 
	\begin{equation}\label{delta v_t + v solves}
		\left\{
		\begin{aligned}
			&\mathcal{L}[\delta v_t + v] = \delta h_t + h=u, 
			&(t,x) &\in (0,T_{\max}) \times \Omega,\\
			&\nabla (\delta v_t + v) \cdot \mathbf{n} = 0,
			&(t,x) &\in (0,T_{\max}) \times \partial\Omega,\\
			&(\delta v_t + v)(0) = \frac{\delta}{\tau}\left(h_0 + \Delta v_0 - v_0\right) + v_0,
			&x &\in \Omega.  
		\end{aligned}
		\right.
	\end{equation}
	Finally, set $\rho$ be the solution to the following heat equation:
	\begin{equation}\label{rho solves}
		\left\{
		\begin{aligned}
			&\mathcal{L}[\rho] = 0, 
			&(t,x) &\in (0,\infty) \times \Omega,\\
			&\nabla \rho \cdot \mathbf{n} = 0,
			&(t,x) &\in (0,\infty) \times \partial\Omega,\\
			&\rho(0) = w_0 - \frac{\delta}{\tau}\left(h_0 + \Delta v_0 - v_0\right) - v_0,
			&x &\in \Omega.
		\end{aligned}
		\right.
	\end{equation}
	The boundedness of $\rho$ follows from  the parabolic maximum principle such that
	\begin{equation}\label{rho upper bound}
		\sup\limits_{t\geq0} \|\rho\|_{L^\infty(\Omega)} \leq \|\rho(0)\|_{L^\infty(\Omega)} \leq C\left(\|u_0\|_{L^\infty(\Omega)} + \|h_0\|_{L^\infty(\Omega)} + \|v_0\|_{W^{2,\infty}(\Omega)}\right) \leq C.
	\end{equation}
	
	After the above preparations,  we follow the ideas in \cite{fujiejiang2021CVPDEcomparison}, \cite{jiangLaurencot2024globalarXiv} and \cite{xiaojiang2024prevention} to derive two key identities which will play the key roles of obtaining the upper bounds for $v$ and $u$.  The first identity is obtained by taking $\mathcal{A}^{-1}$ on both sides of \eqref{Eq of u} as follows.
	\begin{lemma}\label{Lemma: key identity 1}
		For $(t,x) \in (0,T_{\max})\times\Omega$, the following identity holds:
		\begin{equation}\label{key identity 1}
			w_t + \varphi + G(u) = \mathcal{A}^{-1}[\varphi].
		\end{equation}
	\end{lemma}
	Combining the above identity with the definitions of the auxiliary functions $\Psi$, $\psi$, $g$, $\rho$ and equation \eqref{delta v_t + v solves}, we arrive at the second key identity, which serves as a decomposition formula of $v$.
	\begin{lemma}\label{Lemma: key identity 2}
		For $(t,x) \in (0,T_{\max})\times\Omega$, the following identity holds:
		\begin{equation}\label{key identity 2}
			w + \tau\Psi + \tau g = \tau\psi + \delta v_t + v + \rho.
		\end{equation}
	\end{lemma}
	\begin{proof}
		Multiplying \eqref{key identity 1} by $\tau$ and adding $u=\mathcal{A}[w]$ to both sides leads to
		\begin{equation*}
			\tau w_t + \mathcal{A}[w] + \tau\varphi + \tau G(u) = \tau\mathcal{A}^{-1}[\varphi] + u.
		\end{equation*}
		Recalling \eqref{Def: w}, \eqref{Psi solves}, \eqref{psi solves}, \eqref{g solves} and \eqref{delta v_t + v solves}, we can reformulate the above identity as
		\begin{equation*}
			\mathcal{L}[w + \tau\Psi + \tau g] = \mathcal{L}[\tau\psi + \delta v_t + v].
		\end{equation*}
		With $\mathcal{L}[\rho] = 0$ and $\rho(0) = w_0 - \delta v_t(0) - v_0$ from \eqref{rho solves}, it follows that
		\begin{equation*}
			\left\{
			\begin{aligned}
				&\mathcal{L}[w + \tau\Psi + \tau g - \tau\psi - \delta v_t - v - \rho] = 0, 
				&(t,x) &\in (0,T_{\max}) \times \Omega,\\
				&\nabla (w + \tau\Psi + \tau g - \tau\psi - \delta v_t - v - \rho) \cdot \mathbf{n} = 0,
				&(t,x) &\in (0,T_{\max}) \times \partial\Omega,\\
				&(w + \tau\Psi + \tau g - \tau\psi - \delta v_t - v - \rho)(0) = 0,
				&x &\in \Omega.  
			\end{aligned}
			\right.
		\end{equation*}
		Then the identity \eqref{key identity 2} follows from the uniqueness of solutions to the heat equation. This completes the proof.
	\end{proof}

	\section{Global existence for homogeneous case}\label{section 4}
	
	In this section, we establish the global existence of solutions to the homogeneous version of \eqref{Our Problem}, i.e., $f\equiv0$. The proof proceeds in two steps. First, we derive an upper bound for $v$ using the auxiliary functions constructed in the previous section together with the bound on $w$. Then, by a comparison argument along with the second key identity, we obtain an upper bound for $u$, and hence the boundedness of $h$ follows immediately.

	\subsection{Upper bound for $v$}
	
	As a first step, we establish the pointwise upper bound for $w$ based on key identity \eqref{key identity 1}. 
	\begin{lemma} \label{Lemma: w pointwise upper bound}
		Assume that $f\equiv0$.
		Suppose  $\gamma(\cdot)$ satisfies \eqref{A1} and denote $\gamma^*=\gamma(0)$. It holds for all $(t,x) \in [0,T_{\max})\times\Omega$ that
		\begin{equation}\label{pointwise upper bound for w}
			w(t,x) \leq w_0e^{\gamma^*t}.
		\end{equation}
	\end{lemma}
	\begin{proof}
		For the homogeneous case, we rewrite the first key identity \eqref{key identity 1} as 
		\begin{equation}\label{key identity 1 homo-case}
			w_t + \varphi = \mathcal{A}^{-1}[\varphi].
		\end{equation}
		We then infer from the nonnegativity of $\varphi$, together with the monotonicity of $\gamma$ and the elliptic comparison principle, that 
		\begin{equation*}
			w_t \leq \mathcal{A}^{-1}[\varphi] = \mathcal{A}^{-1}[u\gamma(v)] \leq \mathcal{A}^{-1}[\gamma^*u]=\gamma^*\mathcal{A}^{-1}[u]= \gamma^*w.
		\end{equation*}
		By Gronwall’s inequality we obtain \eqref{pointwise upper bound for w}, thereby the proof is complete.
	\end{proof}
	
	Once the boundedness of $w$ is established, the boundedness of $\psi$ follows directly from its definition.
	\begin{corollary}\label{Cor: psi upper bound}
		Suppose that $\gamma(\cdot)$ satisfies \eqref{A1} and $f(\cdot)\equiv0$. Then for any given $T\in(0,T_{\max})$, there exists a constant $\psi^*(T)>0$ depending on $\Omega$, $\gamma$, $T$ and the initial data such that
		\begin{equation}\label{psi upper bound}
			\psi(t,x) \leq \psi^*(T),\quad(t,x) \in [0,T]\times\Omega.
		\end{equation}
	\end{corollary}
	\begin{proof}
		From \eqref{pointwise upper bound for w} we may derive that $w\leq w_0e^{\gamma^*T}:=w^*(T)$ for $(t,x)\in [0,T]\times\Omega$. Then by the elliptic comparison principle we have
		\begin{equation*}
			\mathcal{A}^{-1}[\varphi] = \mathcal{A}^{-1}[\gamma(v)u] \leq \gamma^*\mathcal{A}^{-1}[u] \leq \gamma^*w^*(T).
		\end{equation*}
		Applying the parabolic comparison principle to \eqref{psi solves}, we can deduce that $0\leq\psi \leq \gamma^*w^*(T)$. This completes the proof.
	\end{proof}
	
	Given the upper bound on $w$, we then use the function $\eta$ as an intermediate link to derive an upper bound for $v$.
	
	\begin{proposition}\label{Prop: v upper bound}
		Suppose that $\gamma(\cdot)$ satisfies \eqref{A1} and $f(\cdot)\equiv0$. Then for any given $T\in(0,T_{\max})$, there exists a constant $v^*(T)>0$ depending on $\Omega$, $\delta$, $\tau$, $\gamma$, $T$ and the initial data such that 
		\begin{equation}\label{v: upper bound}
			v(t,x)\leq v^*(T),  \quad (t,x)\in [0,T]\times\Omega.
		\end{equation}
	\end{proposition}
	\begin{proof} On the one hand, from \eqref{pointwise upper bound for w} we derive that $w\leq w_0e^{\gamma^*T}:=w^*(T)$ for $(t,x)\in [0,T]\times\Omega$.
		According to \eqref{Def: w} and \eqref{Def: eta}, we may apply $\mathcal{A}^{-1}$ to both sides of \eqref{Eq of h} to obtain that 
		\begin{equation*}
			\delta\eta_t + \eta = w, \quad (t,x) \in (0,T_{\max})\times\Omega. 
		\end{equation*} 
		Thus,  for $(t,x)\in [0,T]\times\Omega$, it holds that
		\begin{equation*}
			\begin{aligned}
				0\leq\eta(t,x) 
				&=e^{-\frac{1}{\delta}t}\eta_0 + \frac{1}{\delta}\int_{0}^{t}e^{-\frac{t-s}{\delta}}w(s,x)\;\rd s\\
				&\leq e^{-\frac{1}{\delta}t} \|\eta_0\|_{L^\infty(\Omega)} + \frac{w^*(T)}{\delta}\int_{0}^{t}e^{-\frac{t-s}{\delta}}\;\rd s\\
				&\leq e^{-\frac{1}{\delta}t} \|\eta_0\|_{L^\infty(\Omega)} + (1-e^{-\frac{t}{\delta}})w^*(T)\\
				&\leq \max\left\{\|\eta_0\|_{L^\infty(\Omega)}, w^*(T)\right\} \triangleq \eta^*(T).
			\end{aligned}     
		\end{equation*}
		
		On the other hand, we  infer from \eqref{Eq of v} and \eqref{Def: eta} that $\mathcal{L}[v]=h=\mathcal{A}[\eta]$, which gives
		\begin{equation*}
			\mathcal{L}[v-\eta]=-\tau \eta_t=\frac{\tau}{\delta}(\eta-w), \quad (t,x) \in (0,T_{\max})\times\Omega.
		\end{equation*}
		It follows from the parabolic maximum principle that for $(t,x)\in [0,T]\times\Omega$,
		\begin{equation*}
			\begin{aligned}
				\|v-\eta\|_{L^\infty(\Omega)} \leq \|v_0-\eta_0\|_{L^\infty(\Omega)} + \frac{\tau}{\delta}\|\eta-w\|_{L^\infty(\Omega)}\leq \|v_0-\eta_0\|_{L^\infty(\Omega)} +\frac{\tau}{\delta}\left(\eta^*(T)+w^*(T)\right),
			\end{aligned}   
		\end{equation*}
		which then implies desired upper bound for $v$.
	\end{proof}
	\begin{remark} 
		Thanks to the indirect signal production structure, we can establish the upper bound for $v$ without requiring the asymptotic smallness condition $\limsup_{s\to\infty}\gamma(s)<1/\tau$, which is necessary in the direct signal production case \cite{fujiejiang2021CVPDEcomparison}.
	\end{remark}

	\subsection{Upper bound for $u$}
	
	We are now ready to establish an upper bound for $u$ through a direct comparison argument. To this end, we derive the equation for $\varphi=\gamma(v)u$, which takes a dual form to that for $u$, and is particularly amenable to comparison principles.

	Fix $T\in(0,T_{\max})$. Based on the previously obtained bound for $v$ and the properties of $\gamma$, it follows that for all $v\in[0,v^*(T)]$:
	\begin{equation}\label{gamma's bounds}
		0 < \gamma_*(T) \leq \gamma(v) \leq \gamma^* < \infty, \quad \gamma_*(T) \triangleq \gamma\left(v^*(T)\right) > 0,
	\end{equation}
	where $v^*(T)>0$ denotes the upper bound of $v$. As a result,
	\begin{equation}\label{u leq variphi leq u}
		\gamma_*(T)u \leq \varphi=\gamma(v)u \leq \gamma^*u, \quad (t,x)\in[0,T]\times\Omega,
	\end{equation}
	and hence
	\begin{equation}\label{u leq variphi}
		u \leq \frac{\varphi}{\gamma_*(T)}, \quad (t,x)\in[0,T]\times\Omega.
	\end{equation}
	Thus, the problem of bounding $u$ reduces to that of establishing an upper bound for $\varphi$.
	
	\begin{proposition}\label{Prop: u upper bound}
		Suppose that $\gamma(\cdot)$ satisfies \eqref{A1} and $f(\cdot)\equiv 0$. Then for $(t,x)\in [0,T]\times\Omega$, there exists a constant $C(T)>0$ depending on $\Omega$, $\delta$, $\tau$, $\gamma$, $T$ and the initial data such that 
		\begin{equation}\label{u: upper bound}
			u(t,x)\leq C(T).
		\end{equation}
	\end{proposition}
	\begin{proof}
		For the case $f\equiv0$, we rewrite the second key identity as
		\begin{equation}\label{key identity 2 homo-case}
			w + \tau\Psi = \tau\psi + \delta v_t + v + \rho,
		\end{equation}
		where the auxiliary functions $w$, $\varphi$, $\Psi$, $\psi$, $\rho$ are given by \eqref{Def: w}, \eqref{Def: varphi & G(u)}, \eqref{Psi solves}, \eqref{psi solves} and \eqref{rho solves}.
		From  the equation for $u$ and the monotonicity of $\gamma$, we have 
		\begin{equation*}
			\varphi_t
			= \gamma(v)u_t + \gamma'(v)v_tu
			= \gamma(v)\Delta\varphi - \frac{|\gamma'(v)|}{\gamma(v)}\varphi v_t.
		\end{equation*}
		Then with the aid of \eqref{key identity 2 homo-case}, we obtain
		\begin{equation*}
			\begin{aligned}
				\varphi_t 
				& = \gamma(v)\Delta\varphi - \frac{1}{\delta}\frac{|\gamma'(v)|}{\gamma(v)}\left(w + \tau \Psi\right)\varphi + \frac{1}{\delta}\frac{|\gamma'(v)|}{\gamma(v)}\left(\tau \psi + v + \rho\right)\varphi\\
				& \leq \gamma(v)\Delta\varphi + \frac{1}{\delta}\frac{|\gamma'(v)|}{\gamma(v)}\left(\tau \psi + v + \rho\right)\varphi.
			\end{aligned}
		\end{equation*}
		where we have used the fact that $w\geq 0$ and $\Psi\geq 0$. 
		
		Since $v$ is bounded and $\gamma(\cdot)\in C^3\left([0,\infty)\right)$, the quotient $\frac{|\gamma'(v)|}{\gamma(v)}$ remains bounded. 
		Combining with \eqref{rho upper bound}, \eqref{psi upper bound} and \eqref{v: upper bound}, we arrive at
		\begin{equation*}
			\varphi_t \leq \gamma(v)\Delta\varphi + C(T)\varphi.
		\end{equation*}
		Note that $\nabla\varphi\cdot\mathbf{n}=0$ on $\partial\Omega$.    The parabolic comparison principle then yields $\varphi\leq C(T)$, which implies $u\leq \frac{\varphi}{\gamma_*(T)}\leq C(T)$. The proof is thus completed.
	\end{proof}
	
	\noindent \textbf{Proof of Theorem \ref{Th.1.1}.} 
	Under the assumptions of Theorem \ref{Th.1.1}, given any $T\in(0,T_{\max})$, the boundedness of $v$ on $[0,T]$ has been established in Proposition \ref{Prop: v upper bound}. The corresponding boundedness of $u$ then follows from Proposition \ref{Prop: u upper bound}. In view of \eqref{Eq of h}, the estimate
	\begin{equation}\label{h L^infty}
		\begin{aligned}
			h &\leq e^{-\frac{t}{\delta}}\|h_0\|_{L^\infty(\Omega)} + \frac{1}{\delta}\int_0^t e^{-\frac{t-s}{\delta}}\|u\|_{L^\infty(\Omega)} \;\rd s\\
			&\leq e^{-\frac{t}{\delta}}\|h_0\|_{L^\infty(\Omega)} + C(T)(1-e^{-\frac{t}{\delta}})\\
			&\leq C(T),
		\end{aligned}
	\end{equation}
	holds for all $(t,x)\in[0,T]\times\Omega$. And recalling Theorem \ref{Th.2.1}, we deduce that $T_{\max} = \infty$ and thus Theorem \ref{Th.1.1} is proved.\qed
	
	\section{Uniform-in-time boundedness for homogeneous case}\label{section 5}
	In this section, we further assume that $\gamma$ admits a uniform positive lower bound $\gamma_*$, i.e.,
	\begin{equation*}
		0<\gamma_*\leq\gamma(s)\leq \gamma^*=\gamma(0),\quad\forall \,s \in[0,\infty).
	\end{equation*}
	Under the above condition, we establish the uniform-in-time boundedness of solutions to the homogeneous case of \eqref{Our Problem}.
	
	\subsection{Time-independent upper bound for $v$}
	The goal of this section is to establish a uniform-in-time upper bound for $v$, which relies on a corresponding bound for $w$. To this end, we refine the approach developed in \cite{xiaojiang2023JDEglobal} for deriving such a bound for $w$.
	
	\begin{lemma}\label{Lemma: w uniform upper bound}
		Assume that $\gamma(\cdot)$ satisfies \eqref{A1}, \eqref{A2} and that $f\equiv0$. Then we have 
		\begin{equation}
			w(t,x)\leq w^*,\quad \psi(t,x)\leq\psi^*, \quad (t,x)\in[0,\infty)\times\Omega,
		\end{equation}
		with some positive constants $w^*$, $\psi^*$ depending only on $\Omega$, $\gamma$ and the initial data.
	\end{lemma}
	\begin{proof}
		From the first key identity \eqref{key identity 1 homo-case} and definition of $w$, we have
		\begin{equation*}
			w_t + \gamma_*\mathcal{A}[w]=w_t+\gamma_*u\leq w_t+\gamma(v)u
			= \mathcal{A}^{-1}[\gamma(v)u] \leq \gamma^*w, \quad (t,x) \in (0,\infty)\times\Omega.
		\end{equation*}
		Consider the linear problem
		\begin{equation*}
			\left\{
			\begin{aligned}
				&z_t + \gamma_*\mathcal{A}[z] = \gamma^*w, 
				&(x,t) &\in \Omega\times(0,\infty),\\
				&\nabla z \cdot \mathbf{n} = 0, 
				&(x,t) &\in \partial\Omega\times(0,\infty),\\
				&z(x, 0) = w_0(x), 
				&x &\in \Omega,
			\end{aligned}
			\right.
		\end{equation*}
		By the comparison principle, we have
		\begin{equation}\label{w-z comparison}
			0 \leq w \leq z.
		\end{equation}
		The solution to the above equation can be represented as
		\begin{equation*}
			z(t) = e^{-\gamma_* \mathcal{A} t} w_0 + \gamma^* \int_0^t e^{-\gamma_* \mathcal{A}(t - s)} w(s) \;\rd s.
		\end{equation*}
		From the definition of $w$ and Lemma \ref{Lemma:u,v,h L1-estimates}, we obtain 
		\begin{equation*}
			\|w\|_{L^1(\Omega)} = \|u\|_{L^1(\Omega)} = \|u_0\|_{L^1(\Omega)}, \quad t\in[0,\infty).
		\end{equation*}
		Then for $p_1=\frac{N}{N-1}$, it follows
		\begin{equation*}
			\begin{aligned}
				\|z\|_{L^{p_1}(\Omega)}
				&\leq \|w_0\|_{L^{p_1}(\Omega)} + \gamma^* \int_0^t \|e^{-\gamma_* \mathcal{A}(t - s)} w(s)\|_{L^{p_1}(\Omega)} \;\rd s\\
				&\leq C\|w_0\|_{L^{\infty}(\Omega)} + C\gamma^* \int_0^t e^{-\gamma_* (t - s)}(t-s)^{-\frac{N}{2}(1-\frac{1}{p_1})}\|w\|_{L^1(\Omega)} \;\rd s\\
				&\leq C,
			\end{aligned}
		\end{equation*}
		since $\frac{N}{2}(1-\frac{1}{p_1})=\frac12$ and
		$$\int_0^\infty s^{-\frac12}e^{-\gamma_*s} \;\rd s < \infty.$$
		Owing to \eqref{w-z comparison}, we also have 
		\begin{equation*}
			\|w\|_{L^{p_1}(\Omega)} \leq \|z\|_{L^{p_1}(\Omega)} \leq C.
		\end{equation*}
		Now we may use an iteration argument to fix a sequence of increasing numbers $p_j=\frac{N}{N-j}$ ($j=1,2,\dots, N-1$) satisfying
		\begin{equation*}
			\frac{1}{p_j} - \frac{1}{p_{j+1}} =\frac{1}{N} <\frac{2}{N},
		\end{equation*} such that $\|w\|_{L^{p_j}(\Omega)}\leq C(j)$, $j=1,2,\dots,N-1$. Consequently, we have $\|w\|_{L^N(\Omega)}\leq C$.
		Then we can deduce that
		\begin{equation*}
			\begin{aligned}
				\|z\|_{L^{\infty}(\Omega)}
				&\leq \|w_0\|_{L^{\infty}(\Omega)} + \gamma^* \int_0^t \|e^{-\gamma_* \mathcal{A}(t - s)} w(s)\|_{L^{\infty}(\Omega)} \;\rd s\\
				&\leq \|w_0\|_{L^{\infty}(\Omega)} + C\gamma^* \int_0^t e^{-\gamma_* (t - s)}(t-s)^{-\frac{1}{2}}\|w\|_{L^{N}(\Omega)} \;\rd s\\
				&\leq C.
			\end{aligned}
		\end{equation*}
		Thus $0\leq w\leq z\leq C$. In the same manner as done in Corollary \ref{Cor: psi upper bound}, we can further derive  the time-independent upper bound for $\psi$. This completes the proof.
	\end{proof}
	
	As argued in Section \ref{section 4}, we may further derive the uniform-in-time upper bound of $v$.
	
	\begin{proposition}\label{Prop: uniform upper bound for v in homo-case}
		Suppose that $\gamma(\cdot)$ satisfies \eqref{A1} and \eqref{A2} and that $f\equiv0$. Then there exists a constant $v^*>0$ depending on $\Omega$, $\gamma$, $\delta$, $\tau$ and the initial data such that 
		\begin{equation}
			v(t,x)\leq v^*, \quad (t,x)\in[0,\infty)\times\Omega.
		\end{equation}
	\end{proposition}
	\begin{proof}
		It follows from Lemma \ref{Lemma: w uniform upper bound} and the proof of Proposition \ref{Prop: v upper bound} that, for $(t,x)\in [0,\infty)\times\Omega$, 
		\begin{equation}
			0\leq\eta\leq \max\left\{\|\eta_0\|_{L^\infty(\Omega)},w^*\right\}\triangleq\eta^*,
		\end{equation}
		where $w^*$ denotes the uniform-in-time upper bound of $w$. In the same manner as in Proposition \ref{Prop: v upper bound} we deduce that for $(t,x)\in[0,\infty)\times\Omega$,
		\begin{equation}
			v-\eta \leq \|v_0-\eta_0\|_{L^\infty(\Omega)} + \frac{\tau}{\delta}(\eta^* + w^*).
		\end{equation}
		Hence we obtain the time-independent upper bound of $v$. This completes the proof.
	\end{proof}
	
	\subsection{Time-independent upper bound for $u$}
	We now establish a uniform-in-time upper bound for $u$ using a strategy similar to that in Section 4.2. However, since the equation for $\varphi$ contains a nonnegative source term, proving its time-independent boundedness requires additional care. Observing that the equation for $\varphi$ is dual to that for u, we derive energy-type estimates involving $u\varphi^p$, which yield time-independent $L^p$-estimates for $u$.

	By the definition of $\varphi$, we have
	\begin{equation}\label{varphi bounded by u homo-uniform-case}
		\gamma_*u \leq \varphi \leq \gamma^*u, \quad (t,x)\in[0,\infty)\times\Omega.
	\end{equation}We begin with the establishment of low-order estimates.
	\begin{lemma}\label{Lemma: u L2 homo-uniform-case}
		Suppose that $\gamma(\cdot)$ satisfies \eqref{A1} and \eqref{A2} and that $f\equiv0$. Then it holds for $t\in[0,\infty)$ that
		\begin{equation}\label{u L2 homo-uniform-case}
			\|u\|^2_{L^2(\Omega)} \leq C, \quad \int_t^{t+1} \|\nabla\varphi\|^2_{L^2(\Omega)}\;\rd s \leq C,
		\end{equation}
		where $C$ is a positive constant depending on $\Omega$, $\delta$, $\tau$, $\gamma$ and the initial data.
	\end{lemma}
	\begin{proof}
		Recalling that $u = \mathcal{A}[w] = -\Delta w + w$, it holds for all $t\in[0,\infty)$ that
		\begin{equation}\label{w H1 bounded}
			\|w\|^2_{L^2(\Omega)} + \|\nabla w\|^2_{L^2(\Omega)} = \int_\Omega wu\;\rd x\leq w^*\int_\Omega u \;\rd x\leq C.
		\end{equation}
		Multiplying \eqref{key identity 1 homo-case} by $u$ and integrating over $\Omega$, it follows from Young's inequality that
		\begin{equation*}
			\int_\Omega w_tu\;\rd x + \int_\Omega \gamma(v)u^2\;\rd x 
			= \int_\Omega \mathcal{A}^{-1}[\gamma(v)u]u\;\rd x
			=\int_\Omega \gamma(v)u\mathcal{A}^{-1}[u]\;\rd x 
			\leq \gamma^*\int_\Omega uw\;\rd x 
			\leq C.
		\end{equation*}
		Thus we have 
		\begin{equation*}
			\frac{1}{2}\dfrac{d}{dt}\left(\|w\|^2_{L^2(\Omega)} + \|\nabla w\|^2_{L^2(\Omega)}\right) + {\gamma_*}\|u\|^2_{L^2(\Omega)} \leq C,
		\end{equation*}
		An integration of the above inequality with respect to time from $t$ to $t+1$ together with \eqref{w H1 bounded} then yields that
		\begin{equation*}
			\int_t^{t+1} \|u\|^2_{L^2(\Omega)}\;\rd s \leq C,  \quad  t\in[0,\infty).
		\end{equation*}  
		Recall from the proof of Proposition \ref{Prop: u upper bound} that
		\begin{equation*}
			\varphi_t \leq \gamma(v)\Delta\varphi + \frac{1}{\delta}\frac{|\gamma'(v)|}{\gamma(v)}(\tau\psi + v + \rho)\varphi.
		\end{equation*}
		Since $\psi$, $v$ and $\rho$ are all uniform-in-time bounded, it follows that
		\begin{equation}\label{varphi PDI}
			\varphi_t \leq \gamma(v)\Delta\varphi + C\varphi,
		\end{equation}
		where $C$ is a positive constant independent of time. Multiplying \eqref{varphi PDI} by $u$ and integrating over $\Omega$, we have
		\begin{equation*}
			\int_\Omega \varphi_tu\;\rd x \leq \int_\Omega \varphi \Delta\varphi + C\int_\Omega \varphi u\;\rd x.
		\end{equation*}
		Based on \eqref{Eq of u}, it can be deduced that
		\begin{equation}\label{u varphi ODI}
			\dfrac{d}{dt}\int_\Omega \varphi u\;\rd x + 2\|\nabla\varphi\|^2_{L^2(\Omega)} \leq C\int_\Omega \varphi u\;\rd x.
		\end{equation}
		Since 
		\begin{equation*}
			\int_t^{t+1}\int_\Omega \varphi u\;\rd x\;\rd s \leq \gamma^*\int_t^{t+1}\|u\|^2_{L^2(\Omega)}\;\rd s \leq C,
		\end{equation*}
		we may apply the uniform Gronwall inequality Lemma \ref{Lemma: uniform Gronwall inequality} to deduce that there exists some $C>0$ such that
		\begin{equation*}
			\int_\Omega \varphi u\;\rd x \leq C, \quad \text{for all } t \geq 1.
		\end{equation*}
		Furthermore, from \eqref{varphi bounded by u homo-uniform-case} we infer that
		\begin{equation*}
			\|u\|^2_{L^2(\Omega)} \leq \frac{1}{\gamma_*}\int_\Omega \varphi u\;\rd x \leq C, \quad \text{for all } t \geq 1,
		\end{equation*}
		which along with global existence result implies that
		\begin{equation*}
			\|u\|^2_{L^2(\Omega)}\leq C, \quad t\in[0,\infty),
		\end{equation*}
		which in turn gives that 
		\begin{equation*}
			\int_\Omega \varphi u\;\rd x \leq C, \quad t\in[0,\infty).
		\end{equation*}
		Finally, an integration of \eqref{u varphi ODI} over the time interval $(t, t+1)$, together with the above fact, leads to
		\begin{equation*}
			\int_t^{t+1} \|\nabla\varphi\|^2_{L^2(\Omega)}\;\rd s \leq C, \quad t\in[0,\infty).
		\end{equation*}
		This complete the proof.
	\end{proof}
	
	With the above preparations at hand, we now use an iterative method to establish the $L^p$-estimates of $u$ for all $p \in (1, \infty)$.
	
	\begin{proposition}\label{Prop: u Lp homo-case}
		Suppose that $\gamma(\cdot)$ satisfies \eqref{A1} and \eqref{A2} and that $f\equiv0$. Then for all $p\in(1,\infty)$, there exists some constant $C(p)>0$ depending on $\Omega$, $\gamma$, $\delta$, $\tau$, $p$ and the initial data such that 
		\begin{equation}
			\sup\limits_{t\geq0} \|u(t)\|_{L^p(\Omega)}\leq C(p).
		\end{equation}
	\end{proposition}
	\begin{proof}
		As discussed in Lemma \ref{Lemma: u L2 homo-uniform-case}, there exists a time-independent constant $C>0$ such that for all $(t,x)\in(0,\infty)\times\Omega$,
		\begin{equation*}
			\varphi_t \leq \gamma(v)\Delta\varphi + C\varphi.
		\end{equation*}
		Multiplying both sides by $u\varphi^{p-1}$ with some $p>1$ and integrating over $\Omega$, we have
		\begin{equation*}
			\int_\Omega \varphi_t\varphi^{p-1}u\;\rd x \leq \int_\Omega \varphi^p\Delta\varphi\;\rd x + C\int_\Omega \varphi^{p}u\;\rd x,
		\end{equation*}
		which then implies that 
		\begin{equation*}
			\frac{1}{p}\dfrac{d}{dt}\int_\Omega \varphi^p u \;\rd x - \frac{1}{p}\int_\Omega \varphi^p u_t \;\rd x - \int_\Omega \varphi^p\Delta\varphi\;\rd x \leq C\int_\Omega \varphi^{p}u\;\rd x.
		\end{equation*}
		Recalling that $u_t = \Delta\varphi$, an integration by parts applied to the above gives
		\begin{equation}\label{varphi^p u ODI}
			\frac{1}{p}\dfrac{d}{dt}\int_\Omega \varphi^p u \;\rd x + \frac{4}{p+1}\|\nabla\varphi^{\frac{p+1}{2}}\|^2_{L^2(\Omega)} \leq C\int_\Omega \varphi^{p}u\;\rd x.
		\end{equation}
		By the Gagliardo-Nirenberg inequality, 
		\begin{equation*}
			\|\varphi\|^{2+\frac{4}{N}}_{L^{2+\frac{4}{N}}(\Omega)} \leq C\left(\|\nabla\varphi\|^2_{L^2(\Omega)}\|\varphi\|_{L^2(\Omega)}^{\frac{4}{N}} + \|\varphi\|_{L^2(\Omega)}^{2+\frac{4}{N}}\right).
		\end{equation*}
		Combining \eqref{varphi bounded by u homo-uniform-case} and \eqref{u L2 homo-uniform-case}, we may infer that
		\begin{equation*}
			\int_\Omega\varphi^{1+\frac{4}{N}}u\;\rd x \leq \frac{1}{\gamma_*}\|\varphi\|^{2+\frac{4}{N}}_{L^{2+\frac{4}{N}}(\Omega)} \leq C\|\nabla\varphi\|^2_{L^2(\Omega)} + C.
		\end{equation*}
		Integrating the above equality with respect to time from $t$ to $t + 1$ yields
		\begin{equation*}
			\int_t^{t+1}\int_\Omega\varphi^{1+\frac{4}{N}}u\;\rd x\;\rd s \leq C\int_t^{t+1}\|\nabla\varphi\|^2_{L^2(\Omega)}\;\rd s + C \leq C.
		\end{equation*}
		Taking $p_1=1+\frac{4}{N}$ in \eqref{varphi^p u ODI} and then using the uniform Gronwall inequality Lemma \ref{Lemma: uniform Gronwall inequality}, we obtain
		\begin{equation*}
			\int_\Omega\varphi^{p_1}u\;\rd x \leq C(p_1) \quad \text{for all } t\geq1.
		\end{equation*}
		Thanks to \eqref{varphi bounded by u homo-uniform-case}, we infer that
		\begin{equation*}
			\|u\|^{p_1+1}_{L^{p_1+1}(\Omega)} \leq C(p_1) \quad \text{for all } t\geq1,
		\end{equation*}
		which along with the global existence result gives
		\begin{equation*}
			\sup\limits_{t\geq0}\|u\|^{p_1+1}_{L^{p_1+1}(\Omega)} \leq C(p_1).
		\end{equation*}
		By \eqref{varphi bounded by u homo-uniform-case} again, it also holds
		\begin{equation*}
			\sup\limits_{t\geq0} \int_\Omega\left(\varphi^{p_1}u+\varphi^{p_1+1}\right)\;\rd x \leq C(p_1).
		\end{equation*}
		Moreover, an integration of \eqref{varphi^p u ODI} over the time interval $(t, t + 1)$, together with the above fact, leads to
		\begin{equation*}
			\int_t^{t+1} \|\nabla\varphi^{\frac{p_1+1}{2}}\|^2_{L^2(\Omega)}\;\rd s \leq C(p_1).
		\end{equation*}
		Using the Gagliardo–Nirenberg inequality again, we may infer that
		\begin{equation*}
			\begin{aligned}
				\|\varphi\|^{(p_1+1)(1+\frac{2}{N})}_{L^{(p_1+1)(1+\frac{2}{N})}(\Omega)} = \|\varphi^{\frac{p_1+1}{2}}\|^{2+\frac{4}{N}}_{L^{2+\frac{4}{N}}(\Omega)} 
				&\leq C\left(\|\nabla\varphi^{\frac{p_1+1}{2}}\|^2_{L^2(\Omega)}\|\varphi^{\frac{p_1+1}{2}}\|_{L^2(\Omega)}^{\frac{4}{N}} + \|\varphi^{\frac{p_1+1}{2}}\|_{L^2(\Omega)}^{2+\frac{4}{N}}\right)\\
				&\leq C\left(\|\nabla\varphi^{\frac{p_1+1}{2}}\|^2_{L^2(\Omega)}\|\varphi\|_{L^{p_1+1}(\Omega)}^{\frac{2(p_1+1)}{N}} + \|\varphi\|_{L^{p_1+1}(\Omega)}^{(p_1+1)(1+\frac{2}{N})}\right)\\
				&\leq C(p_1)\|\nabla\varphi^{\frac{p_1+1}{2}}\|^2_{L^2(\Omega)} + C(p_1).
			\end{aligned}
		\end{equation*}
		An Integration of the above  over $(t,t+1)$ then gives rise to
		\begin{equation*}
			\int_t^{t+1}\|\varphi\|^{(p_1+1)(1+\frac{2}{N})}_{L^{(p_1+1)(1+\frac{2}{N})}(\Omega)}\;\rd s \leq C(p_1)\int_t^{t+1}\|\nabla\varphi^{\frac{p_1+1}{2}}\|^2_{L^2(\Omega)}\;\rd s + C(p_1) \leq C(p_1).
		\end{equation*}
		Then by taking $p_2 + 1 = (p_1+1)(1+\frac{2}{N})$, we can conclude in the same manner as the above that
		\begin{equation*}
			\sup\limits_{t\geq0}\|u\|^{p_2+1}_{L^{p_2+1}(\Omega)} \leq C(p_2).
		\end{equation*}
		Taking $p_{j+1}+1 = (p_j+1)(1+\frac{2}{N})$, for $j\in\mathbb{N}$ and recalling that $p_1=1+\frac{4}{N}$, we infer that  $$p_{j+1} -p_{j}=\frac2N(p_j+1)\geq\frac2N(2+\frac4N).$$
		Finally, an iteration of the above procedure concludes the proof by choosing $j$ sufficiently large.
	\end{proof}
	
	Next we show the $L^\infty$-estimates for $\nabla v$.
	\begin{corollary}\label{cor: nabla v-estimates homo-case}
		Suppose that the initial data satisfies \eqref{A0}, $\gamma(\cdot)$ satisfies \eqref{A1} and \eqref{A2} and that $f\equiv0$. Then there exists a constant $C>0$ depending on $\Omega$, $\gamma$, $\delta$, $\tau$ and the initial data such that 
		\begin{equation}
			\sup\limits_{t\geq0} \|\nabla v(t)\|_{L^\infty(\Omega)}\leq C.
		\end{equation}
	\end{corollary}
	\begin{proof}
		Multiplying \eqref{Eq of h} by $h^{p-1}$ and integrating over $\Omega$ by parts, we have\
		\begin{equation*}
			\frac{\delta}{p}\dfrac{d}{dt}\int_\Omega h^p \;\rd x + \int_\Omega h^p\;\rd x = \int_\Omega uh^{p-1}\;\rd x.
		\end{equation*}
		Applying Young's inequality, it follows that
		\begin{equation*}
			\frac{\delta}{p}\dfrac{d}{dt}\int_\Omega h^p \;\rd x + \int_\Omega h^p\;\rd x \leq \frac{1}{p}\int_\Omega u^p\;\rd x + \left(1-\frac{1}{p}\right)\int_\Omega h^p\;\rd x.
		\end{equation*}
		Thus we deduce that
		\begin{equation*}
			\delta\dfrac{d}{dt}\int_\Omega h^p\;\rd x + \int_\Omega h^p\;\rd x \leq \int_\Omega u^p\;\rd x.
		\end{equation*}
		With the aid of Proposition \ref{Prop: u Lp homo-case}, we can conclude that  $\|h(t)\|_{L^p(\Omega)}\leq C(p)$ for all $t\geq 0$ and $1<p<\infty$.
		Recalling that $v_0\in W^{2,\infty}(\Omega)$ and $\nabla v_0\cdot \mathbf{n} = 0$ on $\partial\Omega$, we may infer from \eqref{Eq of v} that 
		\begin{equation*}
			v-v_0 = \frac{1}{\tau}\int_0^t e^{-\frac{\mathcal{A}}{\tau}(t-s)}\left(h + \Delta v_0 - v_0\right)\;\rd s, \quad (t,x)\in(0,\infty)\times\Omega.
		\end{equation*}
		According to \cite[Lemma 1.3]{winkler2010JDEaggregation}, for $t\in(0,\infty)$, we have
		\begin{equation*}
			\|\nabla(v-v_0)\|_{L^\infty(\Omega)} \leq C\int_0^t (1 + t^{-\frac{1}{2}-\frac{N}{2p}})e^{-\frac{1+\lambda_1}{\tau}(t-s)}\|h + \Delta v_0 - v_0\|_{L^p(\Omega)}\;\rd s \leq C,
		\end{equation*}
		where $\lambda_1$ denotes the first nonzero eigenvalue of $-\Delta$ in $\Omega$ under Neumann boundary conditions and $p>N$ is chosen such that the above inequality holds. Thus the proof is completed.
	\end{proof}

	\noindent \textbf{Proof of Theorem \ref{Th.1.2}.} 
	A time‑independent upper bound for $v$ has been established in Proposition \ref{Prop: uniform upper bound for v in homo-case}. With the aid of Corollary \ref{cor: nabla v-estimates homo-case}, we may further use a standard bootstrap
	argument (cf. \cite[Lemma A.1]{tao2012boundedness}) to prove that 
	\begin{equation*}
		\sup\limits_{t\geq 0} \|u(t)\|_{L^\infty(\Omega)} \leq C.
	\end{equation*}
	Similar to \eqref{h L^infty}, it holds for all $t \geq 0$ that
	\begin{equation}\label{h uniform L^infty homo-case}
		\begin{aligned}
			\|h\|_{L^\infty(\Omega)} &\leq e^{-\frac{t}{\delta}}\|h_0\|_{L^\infty(\Omega)} + \frac{1}{\delta}\int_0^t e^{-\frac{t-s}{\delta}}\|u\|_{L^\infty(\Omega)} \;\rd s\\
			&\leq e^{-\frac{t}{\delta}}\|h_0\|_{L^\infty(\Omega)} + C(1-e^{-\frac{t}{\delta}})\\
			&\leq C.
		\end{aligned}
	\end{equation}
	Theorem \ref{Th.1.2} is thereby proved.\qed

	\section{The critical mass phenomenon for $\gamma(v)=e^{-v}$ in 2D}\label{section 6}
	In this section, we are interested in the boundedness of the obtained global classical solutions to the homogeneous system in a smooth bounded domain $\Omega \subset \mathbb{R}^2$. We focus on the specific case with  $\gamma(v) = e^{-v}$. For simplicity, we set $\tau = \delta = 1$ and consider the following initial-boundary value problem subject to Neumann boundary conditions:
	\begin{equation}\label{Blow-up Problem}
		\left\{
		\begin{aligned}
			&u_t = \Delta(u e^{-v}), 
			& (t,x) &\in (0,\infty)\times\Omega, \\
			&v_t - \Delta v + v = h, 
			& (t,x) &\in (0,\infty)\times\Omega, \\
			&h_t + h = u,
			& (t,x) &\in (0,\infty)\times\Omega, \\
			&\nabla u\cdot\mathbf{n} = \nabla v\cdot\mathbf{n} = 0, 
			& (t,x) &\in (0,\infty)\times\partial\Omega,\\
			&(u,v,h)(0,x) = (u_0,v_0,h_0), & x &\in \Omega.
		\end{aligned}
		\right.
	\end{equation}
	For $m > 0$, we set 
	\begin{equation*}
		\mathcal{I}_m := \left\{(u,v,h)\in W^{2,\infty}_B(\Omega)\times W^{2,\infty}_B(\Omega)\times W^{1,\infty}(\Omega)\,\, | \,\,u,v,h\geq 0,u\not\equiv0,\|u\|_{L^1(\Omega)}=m\right\},
	\end{equation*}
	with 
	\begin{equation*}
		W^{2,\infty}_B(\Omega) \triangleq \left\{ z\in W^{2,\infty}(\Omega)\,\, | \,\, \nabla z\cdot\textbf{n}=0 \right\}.
	\end{equation*}
	We begin with the availability of a Lyapunov
	functional.
	\begin{lemma}\label{Lemma: Lyapunov functional}
		For $t\geq 0$, there holds 
		\begin{equation}\label{Fuctional and Dissipation}
			\dfrac{d}{dt}\mathcal{F}(u,v,h) + \mathcal{D}(u,v,h) = 0,
		\end{equation}
		where the functional $\mathcal{F}(u,v,h)$ is defined by 
		\begin{equation}\label{Def: functional}
			\mathcal{F}(u,v,h) := \int_\Omega (u\ln{u}-uv)\;\rd x +\frac{1}{2}\left(\|v\|^2_{L^2(\Omega)} + \|\nabla v\|^2_{L^2(\Omega)}\right) + \frac{1}{2}\|\Delta v - v + h\|^2_{L^2(\Omega)},
		\end{equation}
		and 
		\begin{equation}
			\mathcal{D}(u,v,h) := \int_\Omega ue^{-v}|\nabla(\ln{u}-v)|^2\;\rd x + \|\nabla(\Delta v - v + h)\|^2_{L^2(\Omega)} + 2\|\Delta v - v + h\|^2_{L^2(\Omega)}.
		\end{equation}
	\end{lemma}
	\begin{proof}
		Multiplying the first equation of \eqref{Blow-up Problem} by $\ln{u}-v$ and integrating over $\Omega$ by parts, we obtain
		\begin{equation*}
			\begin{aligned}
				\int_\Omega u_t(\ln{u}-v)\;\rd x
				&= -\int_\Omega e^{-v}(\nabla u -u\nabla v)\cdot\nabla(\ln{u}-v) \;\rd x\\
				&= -\int_\Omega ue^{-v}|\nabla(\ln{u}-v)|^2\;\rd x,
			\end{aligned}
		\end{equation*}
		and
		\begin{equation*}
			\begin{aligned}
				\int_\Omega u_t(\ln{u}-v)\;\rd x
				&= \dfrac{d}{dt}\int_\Omega u(\ln{u}-v)\;\rd x - \int_\Omega (u_t-uv_t)\;\rd x\\
				&=  \dfrac{d}{dt}\int_\Omega u(\ln{u}-v)\;\rd x + \int_\Omega uv_t \;\rd x\\
				&= \dfrac{d}{dt}\int_\Omega u(\ln{u}-v)\;\rd x + \int_\Omega h_tv_t \;\rd x + \int_\Omega hv_t \;\rd x\\
				&= \dfrac{d}{dt}\int_\Omega u(\ln{u}-v)\;\rd x + \int_\Omega (v_{tt}-\Delta v_t+v_t)v_t \;\rd x + \int_\Omega (v_t-\Delta v+v)v_t \;\rd x\\
				&= \dfrac{d}{dt}\int_\Omega u(\ln{u}-v)\;\rd x + \frac{1}{2}\dfrac{d}{dt}\left(\|v_t\|^2_{L^2(\Omega)} + \|v\|^2_{L^2(\Omega)} + \|\nabla v\|^2_{L^2(\Omega)}\right) + \|\nabla v_t\|^2_{L^2(\Omega)} + 2\| v_t\|^2_{L^2(\Omega)}.
			\end{aligned}
		\end{equation*} 
		Combining the two preceding identities and substituting $v_t$ via the second equation of \eqref{Blow-up Problem} completes the proof.
	\end{proof}
	
	\subsection{Uniform-in-time boundedness with sub-critical mass}
	In this part, we  establish uniform-in-time boundedness of classical solutions with sub-critical mass stated as below.
	
	\begin{proposition}\label{Prop: uniform boundedness with sub-critical mass}
		Let
		\begin{equation}
			M =
			\begin{cases}
				8\pi & \text{if }~ \Omega = \{x \in \mathbb{R}^2; \ |x| < R\} \text{ and } (u_0, v_0, h_0) \text{ is radial in } x, \\
				4\pi & \text{otherwise}.
			\end{cases}
		\end{equation}
		If $m \triangleq \int_\Omega u_0 dx < M$ , then the global classical solution $(u, v, h)$ to system \eqref{Blow-up Problem} is uniformly bounded in time, i.e,
		\begin{equation}\label{uniform boundedness with sub-critical mass}
			\sup_{t \geq 0} \left( \|u(t)\|_{L^\infty(\Omega)} + \|v( t)\|_{L^\infty(\Omega)}+ \|h( t)\|_{L^\infty(\Omega)} \right) < \infty.
		\end{equation}
	\end{proposition}
	
	Since the functional $\mathcal{F}(u,v,h)$ is the same as that in \cite{laurencot2019DCDSBglobal}, we may recall \cite[Section 3]{laurencot2019DCDSBglobal} and derive the upper and lower bounds on $\mathcal{F}(u,v,h)$.
	
	\begin{lemma}
		Assume that $m < M$. There is a constant $C>0$ such that, for all $t\geq 0$,
		\begin{equation}\label{v H1,v_t L2,functional bounded}
			\|v\|_{L^2(\Omega)} + \|\nabla v\|_{L^2(\Omega)} + \|v_t\|_{L^2(\Omega)} + |\mathcal{F}(u,v,h)| \leq C.
		\end{equation}
	\end{lemma}
	\begin{proof}
		According to Lemma \ref{Lemma: Lyapunov functional}, together with the nonnegativity of $\mathcal{D}$, we can easily derive that
		\begin{equation}\label{functional: upper bound}
			\mathcal{F}(u,v,h)\leq \mathcal{F}(u_0,v_0,h_0) < \infty.
		\end{equation}
		Regarding the lower bound, we denote
		\begin{equation}\label{Def: KS functional}
			\mathcal{F}_0(u,v) := \int_\Omega (u\ln{u}-uv)\;\rd x +\frac{1}{2}\left(\|v\|^2_{L^2(\Omega)} + \|\nabla v\|^2_{L^2(\Omega)}\right).
		\end{equation}
		By comparing the definitions of the two functionals, we deduce that
		\begin{equation}\label{functional & KS functional}
			\mathcal{F}(u,v,h) = \mathcal{F}_0(u,v) + \frac{1}{2}\|\Delta v - v + h\|^2_{L^2(\Omega)} \geq \mathcal{F}_0(u,v).
		\end{equation}
		Since the functional $\mathcal{F}_0(u,v)$ is the same as that of the classical Keller--Segel model, we may argue as in the proof of \cite[Lemma 3.4]{NagaiSenbaYoshida1997Application} to obtain that for any $\varepsilon>0$ and $\varepsilon_0>0$,
		\begin{equation}\label{KS functional: lower bound}
			\mathcal{F}_0(u,v) \geq \left[\frac{1}{2}-(\frac{1}{2M}+\varepsilon)(1+\varepsilon_0)^2 m\right]\|\nabla v\|^2_{L^2(\Omega)} + \varepsilon_0\int_\Omega uv \;\rd x + \frac{1}{2}\|v\|^2_{L^2(\Omega)} - C.
		\end{equation} 
		Since $m<M$, we can choose $\varepsilon>0$, $\varepsilon_0>0$ such that
		\begin{equation*}
			\frac{1}{2}-(\frac{1}{2M}+\varepsilon)(1+\varepsilon_0)^2 m > 0.
		\end{equation*}
		Hence we derive the lower bound for $\mathcal{F}(u,v,h)$.
		Combining \eqref{functional: upper bound}, \eqref{functional & KS functional} and \eqref{KS functional: lower bound}, we may further infer that 
		\begin{equation*}
			\begin{aligned}
				&\min\left\{\frac{1}{2}-(\frac{1}{2M}+\varepsilon)(1+\varepsilon_0)^2 m, \frac{1}{2}\right\}\left(\|v\|^2_{L^2(\Omega)} + \|\nabla v\|^2_{L^2(\Omega)} + \|\Delta v-v+h\|^2_{L^2(\Omega)}\right) \\
				&\quad \leq \mathcal{F}(u,v,h) + C \leq \mathcal{F}(u_0,v_0,h_0) + C.
			\end{aligned}  
		\end{equation*}
		The proof is completed by substituting the second equation of \eqref{Blow-up Problem} into the preceding expression.
	\end{proof}
	
	With the foregoing preparations, we now proceed to prove the uniform boundedness of the solution. We begin with establishing the uniform boundedness of $v$, adopting the approach in \cite[Section 6.1]{fujiejiang2021CVPDEcomparison}.
	
	\begin{lemma}\label{subcritical mass: boundedness of v}
		Assume that $m<M$. There exists $C>0$ depending on $\Omega$ and the initial data such that 
		\begin{equation}\label{subcritical mass v}
			v(x,t) \leq v^*, \quad (t,x) \in [0,\infty)\times\Omega.
		\end{equation}
	\end{lemma}
	\begin{proof}
		Thanks to Lemma \ref{Lemma:u,v,h L1-estimates}, Lemma \ref{Lemma: Moser-Trudinger inequality} and \eqref{v H1,v_t L2,functional bounded}, we infer that for any $k>0$
		\begin{equation*}
			\int_\Omega e^{kv} \;\rd x  \leq C,
		\end{equation*}
		which allows us to argued as the proof of \cite[Lemma 18]{fujiejiang2021CVPDEcomparison} to obtain the uniform-in-time boundedness of $w$. Consequently, we may derive the uniform-in-time upper bound for $v$ by following the approach outlined in Section \ref{section 5}.
	\end{proof}

	\noindent \textbf{Proof of Proposition \ref{Prop: uniform boundedness with sub-critical mass}.}  Given that $m<M$, the uniform-in-time boundedness of $v$ on $[0,\infty)$ has been established in Lemma \ref{subcritical mass: boundedness of v}. We can therefore deduce that $e^{-v^{*}}\leq \gamma(v)=e^{-v}\leq 1$. The corresponding boundedness of $u$ and $h$ then follows from Theorem \ref{Th.1.2}. Thus the proof is completed.\qed
	
	\subsection{Unboundedness with super-critical mass}
	
	In this part, we prove  unboundedness of classical solution with some super-critical mass following the approach in \cite{Horstmann2001EJAMblow-up}. To begin with, we show the compactness of bounded trajectories.
	
	\begin{proposition}\label{Prop: stationary solution}
		Assume that there is a constant $\Lambda >0$ such that
		\begin{equation}\label{bounded solutions}
			u(t) + v(t) + h(t) \leq \Lambda, \quad t\geq 0.
		\end{equation}
		Then there exists a sequence $\{t_k\}_{k\geq 1} \subset (0, \infty)$ with $\lim\limits_{k \to \infty} t_k = \infty$ and corresponding nonnegative solutions $(u_s, v_s,h_s)$ with $m = \|u_0\|_{L^1(\Omega)}$ such that
		\begin{subequations}
			\begin{align}
				&\lim\limits_{k \to \infty} (u(t_k), v(t_k),h(t_k)) = (u_s, v_s,h_s) \quad \text{in } \left(C^2(\overline{\Omega})\right)^2\times C^0(\overline{\Omega}),\label{convergence to the stationary solution}\\
				&\lim\limits_{k \to \infty} \mathcal{F}(u(t_k),v(t_k),h(t_k)) = \mathcal{F}(u_s,v_s,h_s),\label{functional converge to steady state}
			\end{align}
		\end{subequations}
		where 
		\begin{equation}\label{stationary solution u_s,v_s,h_s}
			u_s = h_s =  \frac{m e^{v_s}}{\int_\Omega e^{v_s} \;\rd x}, \quad v_s\geq 0,
		\end{equation}
		and $v_s$ solves
		\begin{equation}\label{stationary v_s solves}
			\left\{
			\begin{aligned}
				&-\Delta v_s + v_s =  \frac{m e^{v_s}}{\int_\Omega e^{v_s} \;\rd x} 
				&\text{in}\,&\Omega,\\
				&\nabla v_s\cdot\mathbf{n} = 0
				&\text{on}\,&\partial\Omega.
			\end{aligned}
			\right.
		\end{equation}
		In other words, $(u_s,v_s,h_s)$ is a stationary solution  of \eqref{Blow-up Problem}.
	\end{proposition}
	\begin{proof} In order to prove the compactness, we need to derive some dissipative energy estimates.   From the energy-dissipation relation \eqref{Fuctional and Dissipation} and the second equation of \eqref{Blow-up Problem}, we infer that
		\begin{equation*}
			\int_0^t \int_\Omega ue^{-v}|\nabla(\ln{u}-v)|^2\;\rd x \;\rd s + \int_0^t\|\nabla v_t\|^2_{L^2(\Omega)}\;\rd s + 2\int_0^t\|v_t\|^2_{L^2(\Omega)}\;\rd s = \mathcal{F}(u_0,v_0,h_0) - \mathcal{F}(u,v,h) ,
		\end{equation*}
		for $t\geq 0$. Owing to the boundedness of $\mathcal{F}(u,v,h)$ and $\mathcal{F}(u_0,v_0,h_0)$, we obtain that
		\begin{equation}\label{v_t diss}
			\int_0^\infty \int_\Omega ue^{-v}|\nabla(\ln{u}-v)|^2\;\rd x \;\rd s<\infty; \quad \int_0^\infty\|v_t\|^2_{L^2(\Omega)}\;\rd s<\infty.
		\end{equation}
		Recalling $\varphi=e^{-v}u$ in the present setting, one verifies that
		\begin{equation*}
			\varphi_t = e^{-v}\Delta\varphi - \varphi v_t.
		\end{equation*}
		Multiplying the above equation by $-\Delta\varphi$ and integrating over $\Omega$ by parts, we obtain
		\begin{equation*}
			\frac{1}{2}\dfrac{d}{dt}\|\nabla\varphi\|^2_{L^2(\Omega)} + \int_\Omega e^{-v}|\Delta\varphi|^2\,\rd x = \int_\Omega\varphi\Delta\varphi v_t\,\rd x.
		\end{equation*}
		By virtue of \eqref{bounded solutions}, we have $e^{-\Lambda}\leq e^{-v}\leq 1$ and $\varphi<\Lambda$. By Young's inequality, we may therefore conclude that
		\begin{equation*}
			\frac{1}{2}\dfrac{d}{dt}\|\nabla\varphi\|^2_{L^2(\Omega)} + e^{-\Lambda} \|\Delta\varphi\|^2_{L^2(\Omega)} 
			\leq\Lambda\int_\Omega|\Delta\varphi v_t|\,\rd x
			\leq \frac{e^{-\Lambda}}{2}\|\Delta\varphi\|^2_{L^2(\Omega)} + C(\Lambda)\|v_t\|^2_{L^2(\Omega)}.
		\end{equation*}
		An integration of the above over $(0,t)$ along with \eqref{v_t diss} then gives rise to
		\begin{equation*}
			\|\nabla\varphi(t)\|^2_{L^2(\Omega)} + e^{-\Lambda}\int_0^t \|\Delta\varphi\|^2_{L^2(\Omega)}\,\rd s \leq C(\Lambda)\int_0^t \|v_t\|^2_{L^2(\Omega)}\,\rd s + \|\nabla\varphi(0)\|^2_{L^2(\Omega)} <\infty \quad \text{for all }t\geq 0.
		\end{equation*}
		Thus, recalling that $u_t=\Delta\varphi$, we obtain that
		\begin{equation*}
			\int_0^\infty \|u_t\|^2_{L^2(\Omega)}\,\rd s = \int_0^\infty \|\Delta\varphi\|^2_{L^2(\Omega)}\,\rd s <\infty.
		\end{equation*}
		Next, differentiating the equation for $h$ with respect to $t$, then multiplying the resultant by $h_t$, and integrating over $\Omega$, we can apply H\"older's inequality to derive that
		\begin{equation*}
			\frac{1}{2}\dfrac{d}{dt}\|h_t\|^2_{L^2(\Omega)} + \|h_t\|^2_{L^2(\Omega)} = \int_\Omega u_t h_t\,\rd x \leq \frac{1}{2}\|h_t\|^2_{L^2(\Omega)} + \frac{1}{2}\|u_t\|^2_{L^2(\Omega)}.
		\end{equation*}
		An integration over the time interval $(0,t)$ yields that
		\begin{equation*}
			\|h_t(t)\|^2_{L^2(\Omega)} + \int_0^t\|h_t\|^2_{L^2(\Omega)}\,\rd s \leq \int_0^t\|u_t\|^2_{L^2(\Omega)}\,\rd s + \|h_t(0)\|^2_{L^2(\Omega)} <\infty \quad \text{for all }t\geq0.
		\end{equation*}
		Thus we obtain
		\begin{equation*}
			\int_0^\infty \|h_t\|^2_{L^2(\Omega)}\,\rd s <\infty.
		\end{equation*}
		
		With the above preparations, we are ready to prove the convergence result. First, from the boundedness of $(u, v, h)$ and parabolic Schauder theory, we infer that $(u(t))_{t>0}$ and $(v(t))_{t>0}$ are relatively compact in $C^2(\overline{\Omega})$, whereas $(h(t))_{t>0}$ is relatively compact in $C^0(\overline{\Omega})$. Hence, for $\{t_k\}_{k\geq 1}$ with $\lim\limits_{k \to \infty} t_k = \infty$, one may extract a subsequence, still denoted by $\{t_k\}$ such that $(u(t_k),v(t_k),h(t_k))$ converges to some $(u_s,v_s)$ in $C^2(\overline{\Omega})$ and that $h(t_k)$ converges to some $h_s$ in $C^0(\overline{\Omega})$. Moreover, preceding dissipative estimates implies that
		\begin{equation}\label{v_t(t_k)}
			\|u_t(t_k)\|^2_{L^2(\Omega)}+
			\|v_t(t_k)\|^2_{L^2(\Omega)} +\|h_t(t_k)\|^2_{L^2(\Omega)}\to 0  \quad  \text{as } k\to\infty,
		\end{equation}
		which allows us to pass to the limit to conclude that $(u_s,v_s,h_s)$ satisfies \eqref{stationary solution u_s,v_s,h_s} and \eqref{stationary v_s solves}. This completes the proof.\end{proof}

	The remainder of this section is devoted to showing that there exists nonnegative initial data $(u_0, v_0, h_0)\in \mathcal{I}_m$ with $m\in (8\pi,\infty)\setminus4\pi \mathbb{N}$ such that the corresponding solution to \eqref{Blow-up Problem} blows up at time infinity. For given $m>0$, we put 
	\begin{equation}
		\mathcal{S}_m := \left\{(u,v,h)\in C^2(\overline{\Omega})\times C^2(\overline{\Omega})\times C^0(\overline{\Omega}) \,| \,(u,v,h)~\text{satisfy \eqref{stationary solution u_s,v_s,h_s} and \eqref{stationary v_s solves}}\right\}.
	\end{equation}
	Equivalently, $\mathcal{S}_m$ is the set of nonnegative stationary solutions $(u, v, h)$ to \eqref{Blow-up Problem} which belong to $C^2(\overline{\Omega})\times C^2(\overline{\Omega})\times C^0(\overline{\Omega})$ and for which $\int_\Omega u\;\rd x = m$. As argued in \cite[Section 4]{laurencot2019DCDSBglobal}, we recast the problem in a way suitable to apply results from \cite{Horstmann2001EJAMblow-up}. The process begins with obtaining a lower bound of the Lyapunov functional on $\mathcal{S}_m$ for suitable values of $m$. As established in \cite[Proposition 4.1]{laurencot2019DCDSBglobal}, we have the following proposition.
	
	\begin{proposition}\label{Prop: stationary solutions' functional bounded below}
		(a) Assume $m\in (4\pi,\infty)\setminus4\pi \mathbb{N}$. Then 
		\begin{equation*}
			\lambda_m := \inf\limits_{(u,v,h) \in \mathcal{S}_m} \mathcal{F}(u,v,h) > -\infty.
		\end{equation*}
		(b) Assume that $\Omega=B_R(0)$ for some $R>0$, $m\in (8\pi,\infty)$. Then 
		\begin{equation*}
			\lambda_m := \inf\limits_{(u,v,h) \in \mathcal{S}^{rad}_m} \mathcal{F}(u,v,h) > -\infty,
		\end{equation*}
		where $\mathcal{S}^{rad}_m:=\left\{\right(u,v,h)\in\mathcal{S}_m \, | \, u,v,h  \,\, \text{are radially symmetric}\}$.
	\end{proposition}
	
	We are now in a position to prove that for some initial data on the set $\mathcal{I}_m$, the functional $\mathcal{F}(u,v,h)$ is not bounded below.
	
	\begin{proposition}\label{Prop: Blow-up solutions}
		Assume $m\in(8\pi,\infty)\setminus4\pi \mathbb{N}$. Then 
		\begin{equation*}
			\inf\limits_{(u,v,h) \in \mathcal{I}_m} \mathcal{F}(u,v,h) = -\infty.
		\end{equation*}
	\end{proposition}
	\begin{proof}
		Motivated by the construction in \cite{fujieSenba2019JDEblowup}, we follows the proof in \cite{fujiejiang2021CVPDEcomparison} to consider the following basis functions: 
		\begin{equation*}
			\overline{u}_\lambda(x) := \frac{8\lambda^2}{(1+\lambda^2|x|^2)^2},\quad 
			\overline{v}_{\lambda,r}(x) := 2\ln{\frac{1 + \lambda^2r^2}{1 + \lambda^2|x|^2}} + \ln{8} \quad \text{for all }x\in\mathbb{R}^2.
		\end{equation*}
		where $\lambda\geq1$ and $r\in(0,1)$. Then for $x\in B_r(0)\triangleq\{x\in\mathbb{R}^2\,;\,|x|<r\}$, it follows that 
		\begin{equation*}
			0 < \overline{u}_\lambda(x) \leq 8\lambda^2,\quad
			\overline{v}_{\lambda,r}(x) > \ln{8} > 0.
		\end{equation*}
		Furthermore, both $\overline{u}_\lambda$ and $\overline{v}_{\lambda,r}$ are smooth in $B_r(0)$.
		In the subsequent analysis, we fix $r\in(0,1)$ and $q\in\Omega$ such that $B_{2r}(q) \triangleq \{x\in\mathbb{R}^2\,;\,|x-q|<2r\} \subset \Omega$. By translation, we may assume that $q=0$. Then we fix $r_1\in(0,r)$, which allows us to construct a smooth radially symmetric cutoff function $\phi_{r,r_1}$ satisfying
		\begin{equation*}
			\phi\left(B_{r_1}(0)\right)=1,\quad
			0\leq\phi\leq1,\quad
			\phi\left(\mathbb{R}^2\setminus B_r(0)\right)=0,\quad
			x\cdot\nabla\phi(x)\leq 0,
		\end{equation*}
		where $B_{r_1}(0)\triangleq\{x\in\mathbb{R}^2\,;\,|x|<r_1\}$.
		Clearly, $\phi\in C_c^\infty(B_r(0))$ and $D^\alpha\phi\in C_c^\infty(B_r(0))$ for all multi-indices $\alpha\in\mathbb{N}^2$.
		Now we define 
		\begin{equation*}
			u_0 := a\overline{u}_\lambda\phi,\quad 
			v_0 := a\overline{v}_{\lambda,r}\phi,
		\end{equation*}
		with some $a>1$. As is easily verified, $u_0,v_0\in C_c^\infty(B_r(0))$. In particular, we have $u_0,v_0\in W_B^{2,\infty}(\Omega)$ and $u_0,v_0\geq 0$.
		As shown in \cite[Lemma 20]{fujiejiang2021CVPDEcomparison}, one can choose
		\begin{equation*}
			a = a(\lambda) \in\left[\frac{m}{8\pi}, \frac{m(1+r_1^2)}{8\pi r_1^2}\right],
		\end{equation*}
		such that $\int_\Omega u_0\;\rd x=m$.
		Then it is proved in \cite[Lemma 21-22]{fujiejiang2021CVPDEcomparison} that there exists $C>0$ such that for all $\lambda\geq 1$, the following estimates are valid:
		\begin{equation*}
			\int_\Omega u_0\ln{u_0}\;\rd x \leq 16a\pi\ln{\lambda} + C,
		\end{equation*}
		\begin{equation*}
			\int_\Omega u_0v_0\;\rd x \geq 32a^2\pi\ln{\lambda} - C.
		\end{equation*}In addition, it is shown in \cite[Lemma 22]{fujiejiang2021CVPDEcomparison} that, for any $\varepsilon_1>0$, there also exists $C(\varepsilon_1)>0$ such that
		\begin{equation*}
			\frac{1}{2}\left(\|v_0\|^2_{L^2(\Omega)} + \|\nabla v_0\|^2_{L^2(\Omega)}\right)\leq 16(1+\varepsilon_1)a^2\pi\ln{\lambda} + C(\varepsilon_1).
		\end{equation*}
		By the definition of $\mathcal{F}_0$ and the above estimates, it follows that
		\begin{equation*}
			\begin{aligned}
				\mathcal{F}_0(u_0,v_0)
				&\leq 16a\pi\ln{\lambda} - 32a^2\pi\ln{\lambda} + 16(1+\varepsilon_1)a^2\pi\ln{\lambda} + C(\varepsilon_1)\\
				&= -16a\pi\left[(1-\varepsilon_1)a-1\right]\ln{\lambda} + C(\varepsilon_1)\\
				&\leq -2m\left(\frac{m(1-\varepsilon_1)}{8\pi} - 1\right)\ln{\lambda}+ C(\varepsilon_1).
			\end{aligned}
		\end{equation*}
		Fix some small $\varepsilon_1$ independent of $\lambda$ such that 
		\begin{equation*}
			\frac{m(1-\varepsilon_1)}{8\pi} - 1 > 0.
		\end{equation*}
		It then follows that 
		\begin{equation}\label{F_0 unbounded below}
			\mathcal{F}_0(u_0,v_0)\to -\infty\quad \text{as } \lambda\to\infty.
		\end{equation}
		
		Next, we construct a proper initial datum for $h$.
		By direct calculations, we obtain that
		\begin{equation*}
			\nabla \overline{v}_{\lambda,r} = -\frac{4\lambda^2x}{1+\lambda^2|x|^2},\quad 
			\Delta \overline{v}_{\lambda,r} = -\frac{8\lambda^2}{(1+\lambda^2|x|^2)^2}.
		\end{equation*}
		Thus we may infer that
		\begin{equation*}
			\begin{aligned}
				-\Delta v_0 + v_0 
				&= -\left(a\Delta\overline{v}_{\lambda,r}\phi + 2a\nabla\overline{v}_{\lambda,r}\cdot\nabla\phi + a\overline{v}_{\lambda,r}\Delta\phi\right) + a\overline{v}_{\lambda,r}\phi\\
				&= \frac{8a\lambda^2\phi}{(1+\lambda^2|x|^2)^2} + \frac{8a\lambda^2x}{1+\lambda^2|x|^2}\cdot\nabla\phi - a\overline{v}_{\lambda,r}\Delta\phi + a\overline{v}_{\lambda,r}\phi\\
				&\geq -\sup_{B_r(0)\setminus B_{r_1}(0)}\left(\frac{8a\lambda^2|x|}{1+\lambda^2|x|^2}|\nabla\phi|\right) - \sup_{B_r(0)\setminus B_{r_1}(0)}\left [a\left(2\ln{\frac{1 + \lambda^2r^2}{1 + \lambda^2|x|^2}} + \ln{8}\right)|\Delta\phi|\right],
			\end{aligned}
		\end{equation*}
		For $x\in B_r(0)\setminus B_{r_1}(0)$, it holds that
		\begin{equation*}
			\frac{8a\lambda^2|x|}{1+\lambda^2|x|^2} \leq \frac{8a\lambda^2|x|}{\lambda^2|x|^2} \leq \frac{8a}{r_1}, \quad \ln{\frac{1 + \lambda^2r^2}{1 + \lambda^2|x|^2}} \leq \ln{\frac{1 + \lambda^2r^2}{1 + \lambda^2r_1^2}}.
		\end{equation*}
		Since $\ln{\frac{1 + \lambda^2r^2}{1 + \lambda^2r_1^2}}$ is increasing with respect to $\lambda$, which implies that
		\begin{equation*}
			\ln{\frac{1 + \lambda^2r^2}{1 + \lambda^2r_1^2}} \leq \lim\limits_{\lambda\to\infty}\ln{\frac{1 + \lambda^2r^2}{1 + \lambda^2r_1^2}} = 2\ln{\frac{r}{r_1}}.
		\end{equation*}
		We may conclude that 
		\begin{equation*}
			-\Delta v_0 + v_0 \geq -\frac{8a}{r_1}\sup_{B_r(0)\setminus B_{r_1}(0)}|\nabla\phi| - a\left(4\ln{\frac{r}{r_1}} + \ln{8}\right)\sup_{B_r(0)\setminus B_{r_1}(0)}|\Delta\phi| \geq -K,
		\end{equation*}
		where $K$ is a positive constant depending on $r, r_1$ but independent of $\lambda$. Let $h_0 \triangleq -\Delta v_0 + v_0 + K$. Then $h_0\geq 0$, $h_0\in C^\infty(B_r(0))$, and in particular $h_0\in W^{1,\infty}(\Omega)$.
		Therefore, we have
		\begin{equation}\label{Delta v_0 -v_0 +h_0 L2}
			\|\Delta v_0 - v_0 + h_{0}\|^2_{L^2(\Omega)} = K^2|\Omega|.
		\end{equation}
		Recall that 
		$$\mathcal{F}(u,v,h)=\mathcal{F}_0(u,v) + \frac{1}{2}\|\Delta v - v + h\|^2_{L^2(\Omega)},$$
		we may conclude from \eqref{F_0 unbounded below} and \eqref{Delta v_0 -v_0 +h_0 L2} that 
		\begin{equation*}
			\mathcal{F}(u_0,v_0,h_0) =\mathcal{F}_0(u_0,v_0) + \frac{K^2|\Omega|}{2} \to -\infty \quad \text{as } \lambda \to \infty.
		\end{equation*}
		Thus, the proof is complete. 
	\end{proof}
	
	\noindent\textbf{Proof of Theorem \ref{Th.1.3}.} By Proposition \ref{Prop: uniform boundedness with sub-critical mass}, the proof for the first part in Theorem \ref{Th.1.3} concerning bounded solutions is complete. As to the unboundedness part, we consider $m\in(8\pi,\infty)\setminus4\pi\mathbb{N}$. According to Proposition \ref{Prop: Blow-up solutions}, there exists $(u_0,v_0,h_0)\in \mathcal{I}_m$ such that 
	\begin{equation}\label{construct contradiction}
		\mathcal{F}(u_0,v_0,h_0) < \lambda_m,
	\end{equation}
	with $\lambda_m$ being defined in Proposition \ref{Prop: stationary solutions' functional bounded below} (a). Assume that there is $\Lambda>0$ such that 
	\begin{equation*}
		\|u(t)\|_{L^\infty(\Omega)} + \|v(t)\|_{L^\infty(\Omega)} + \|h(t)\|_{L^\infty(\Omega)} \leq \Lambda \quad\text{for all } t\geq 0.
	\end{equation*}
	Then according to Proposition \ref{Prop: stationary solution}, there exists a sequence $(t_k)_{k\geq 1}$ with $\lim\limits_{k\to\infty}t_k=\infty$ such that
	for $(u_s,v_s,h_s)\in\mathcal{S}_m$,
	\begin{equation*}
		\mathcal{F}(u_s,v_s,h_s) = \lim\limits_{k\to\infty}\mathcal{F}((u(t_k),v(t_k),h(t_k)).
	\end{equation*}
	Combining the previous identity with Proposition \ref{Prop: stationary solutions' functional bounded below} we infer that
	\begin{equation*}
		\lambda_m 
		\leq \mathcal{F}(u_s,v_s,h_s) 
		= \lim\limits_{k\to\infty}\mathcal{F}((u(t_k),v(t_k),h(t_k)) 
		\leq \mathcal{F}(u_0,v_0,h_0), 
	\end{equation*}
	which contradicts \eqref{construct contradiction}. Thus we complete the proof. \qed

	\section{Globally boundedness for nonhomogeneous case}\label{section 7}
	In this section, we aim to prove that the presence of an external source with mere superlinear damping ensures the uniform-in-time boundedness of solutions in all space dimensions, requiring no additional assumptions on $\gamma$. Following the same approach as in Section \ref{section 4}, we first use the auxiliary functions from Section \ref{section 3} to bound $v$, and then establish a bound for $u$ via a comparison argument. Crucially, these bounds are uniform in time.
	
	\subsection{Uniform-in-time upper bound for $v$}
	We begin by establishing a uniform-in-time upper bound for $w$, which also yields the corresponding bound for $\psi$; the uniform-in-time upper bound for $v$ then follows immediately.
	
	\begin{lemma}
		Suppose that $\gamma(\cdot)$ satisfies \eqref{A1} and $f(\cdot)$ satisfies \eqref{H}. Then for all $(t,x) \in [0,T_{\max})\times\Omega$, there exist positive constants $w^*$, $\psi^*$ depending on $\Omega$, $\gamma$, $f$ and the initial data such that 
		\begin{equation}\label{w: uniform upper bound inhomo-case}
			w(t,x)\leq w^*,\quad \psi(t,x)\leq\psi^*.
		\end{equation}
	\end{lemma}
	\begin{proof}
		By adding $w$ to both sides of the key identity \eqref{key identity 1}, it follows from \eqref{A1} and \eqref{f(s) inequality} that 
		\begin{align*}
			w_t + w + \varphi + G(u) 
			& = \mathcal{A}^{-1}[\gamma(v)u + u]\\
			& \leq \left(\gamma^* + 1\right)\mathcal{A}^{-1}[u]\\
			& \leq \left(\gamma^* + 1\right)\mathcal{A}^{-1}\left[\frac{uf(u)}{\gamma^* + 1} + \frac{C}{\gamma^* + 1}\right]\\
			& \leq G(u) + C.
		\end{align*}
		Since $\varphi \geq 0$, we have 
		\begin{equation*}
			w_t + w \leq C,
		\end{equation*}
		with $C>0$ depending only on $\gamma$ and $f$. Standard ODE techniques then yield $w \leq w^*$. In the same manner as argued in Corollary \ref{Cor: psi upper bound}, we may further derive the uniform-in-time boundedness of $\psi$ from that of $w$. This completes the proof.
	\end{proof}
	
	Following the argument in the proof of Proposition \ref{Prop: uniform upper bound for v in homo-case}, we obtain a time-independent upper bound for $v$ as well.
	\begin{corollary}\label{Cor: uniform upper bound for v in inhomo-case}
		Suppose that $\gamma(\cdot)$ satisfies \eqref{A1} and $f(\cdot)$ satisfies \eqref{H}. Then for $(t,x)\in [0,T_{\max})\times\Omega$, there exists a constant $v^*>0$ depending on $\Omega$, $\delta$, $\tau$, $\gamma$, $f$ and the initial data such that 
		\begin{equation}\label{v uniform upper bound}
			v(t,x)\leq v^*.
		\end{equation}
	\end{corollary}
	
	
	
	\subsection{Uniform-in-time upper bound for $u$}
	The uniform-in-time boundedness of $v$, together with the monotonicity of $\gamma$, provides the strictly positive upper and lower bounds for $\gamma$ that are also time-independent.
	\begin{equation}\label{gamma's uniform bounds}
		0 < \gamma_* \leq \gamma(v) \leq \gamma^* < \infty, \quad \gamma_* \triangleq \gamma(v^*) > 0,
	\end{equation}
	where $v^*$ denotes the uniform-in-time upper bound of $v$. We thus obtain the following estimates for $u$ and $\varphi$:
	\begin{equation}\label{u<varphi<u}
		\gamma_*u \leq \varphi \leq \gamma^*u, \quad (t,x)\in[0,T_{\max})\times\Omega.
	\end{equation}
	
	%
	
	Thanks to the super-linearity of the external sources, we are now able to derive a time-independent upper bound for $\varphi$ by a direct comparison argument.
	The uniform-in-time boundedness of $u$ then follows from that of $\varphi$.
	
	\begin{proposition}\label{Prop: u upper bound inhomo-case}
		Suppose that $\gamma(\cdot)$ satisfies \eqref{A1} and $f(\cdot)$ satisfies \eqref{H}. Then for $(t,x)\in [0,T_{\max})\times\Omega$, there exists a constant $C>0$ depending on $\Omega$, $\delta$, $\tau$, $\gamma$, $f$ and the initial data such that 
		\begin{equation}
			u(t,x)\leq C.
		\end{equation}
	\end{proposition}
	\begin{proof}
		From the definition of $\varphi$ and the equation for $u$, we have 
		\begin{equation*}
			\varphi_t
			= \gamma(v)u_t + \gamma'(v)v_tu 
			= \gamma(v)\Delta\varphi - \varphi f(u) - \frac{|\gamma'(v)|}{\gamma(v)}\varphi v_t.
		\end{equation*}
		With the aid of key identity \eqref{key identity 2}, we obtain that
		\begin{equation*}
			\begin{aligned}
				\varphi_t
				&= \gamma(v)\Delta\varphi - \varphi f(u) - \frac{1}{\delta}\frac{|\gamma'(v)|}{\gamma(v)}\left(w + \tau \Psi + \tau g\right)\varphi + \frac{1}{\delta}\frac{|\gamma'(v)|}{\gamma(v)}\left(\tau \psi + v + \rho\right)\varphi\\
				&\leq \gamma(v)\Delta\varphi - \varphi f(u) - \frac{\tau}{\delta}\frac{|\gamma'(v)|}{\gamma(v)} g\varphi + \frac{1}{\delta}\frac{|\gamma'(v)|}{\gamma(v)}\left(\tau \psi + v + \rho\right)\varphi,
			\end{aligned}
		\end{equation*}
		where we have used the fact that $w\geq 0$ and $\Psi\geq 0$. 
		Since $v$ is uniformly bounded and $\gamma(\cdot)\in C^3\left([0,\infty)\right)$, the quotient $\frac{|\gamma'(v)|}{\gamma(v)}$ stays uniformly bounded in time. 
		From \eqref{f(s) inequality} and the elliptic comparison principle, it follows that
		\begin{equation*}
			G(u) = \mathcal{A}^{-1}[uf(u)] \geq \mathcal{A}^{-1}[u-C] \geq - C,
		\end{equation*}
		where $C>0$ is a constant independent of time. Applying the parabolic comparison principle to \eqref{g solves}, we then derive that $g$ is uniformly bounded from below.
		With the uniform boundedness of $\rho$, $\psi$ and $v$ established in \eqref{rho upper bound}, \eqref{w: uniform upper bound inhomo-case} and \eqref{v uniform upper bound}, we may infer that
		\begin{equation*}
			\varphi_t + \varphi 
			\leq \gamma(v)\Delta\varphi + C\varphi - \varphi f(u).
		\end{equation*}
		It follows from \eqref{f(s) inequality} and \eqref{u<varphi<u} that
		\begin{equation*}
			\begin{aligned}
				\varphi_t + \varphi 
				&\leq \gamma(v)\Delta\varphi + C\gamma^*u - \gamma_*uf(u)\\
				&\leq \gamma(v)\Delta\varphi + C.
			\end{aligned}
		\end{equation*}
		Thus $\varphi$ satisfies
		\begin{equation*}
			\left\{
			\begin{aligned}
				&\varphi_t - \gamma(v)\Delta\varphi + \varphi \leq C, 
				&(t,x) &\in (0,T_{\max})\times\Omega,\\
				&\nabla\varphi \cdot \mathbf{n} = 0,
				&(t,x) &\in (0,T_{\max})\times\partial\Omega,\\
				&\varphi(0) = \gamma(v_0)u_0,
				&x &\in \Omega.
			\end{aligned}
			\right.
		\end{equation*}
		An application of the parabolic comparison principle gives $\varphi \leq C$, where the constant $C>0$ may depend on $\Omega$, $\delta$, $\tau$, $\gamma$, $f$ and the initial data, but is independent of time. Hence, the conclusion follows from \eqref{u<varphi<u}.
	\end{proof}

	\noindent \textbf{Proof of Theorem \ref{Th.1.4}.} 
	Under the assumptions of Theorem \ref{Th.1.4}, the uniform-in-time boundedness of $v$ on $[0,T_{\max})$ has been established in Corollary \ref{Cor: uniform upper bound for v in inhomo-case}. The corresponding boundedness of $u$ then follows from Proposition \ref{Prop: u upper bound inhomo-case}. And the global $L^\infty$-estimates for $h$ follow as in \eqref{h uniform L^infty homo-case}. By Theorem \ref{Th.2.1}, we note that $T_{\max} = \infty$. Thus Theorem \ref{Th.1.4} is proved.\qed
	
	\section*{Acknowledgments}
	\noindent This work is supported by National Natural Science Foundation of China (NSFC)
	under grants No.~12271505. 
	
	\normalem    
	\bibliographystyle{siam}
	\bibliography{refs}
\end{document}